\documentclass[]{amsart}

\usepackage[margin=1.5in]{geometry}
\usepackage{amsmath,amsthm,amsfonts,amssymb,amscd}
\usepackage{lastpage}
\usepackage{fancyhdr}
\usepackage{mathrsfs}
\usepackage{xcolor}
\usepackage{graphicx}
\usepackage{listings}
\usepackage{hyperref}
\usepackage{enumitem}
\usepackage{relsize}
\usepackage{tikz-cd}
\usepackage{mdwlist}
\usepackage{multicol}
\usepackage{stmaryrd}
\usepackage{float}
\usepackage{adjustbox}
\usepackage{tikz}
\usepackage{bbm}
\usepackage[noabbrev,nameinlink]{cleveref}
\usetikzlibrary{matrix, calc, arrows}

\hypersetup{
    colorlinks=true, 
    linktoc=all,     
    linkcolor=blue,  
}

\setlist[enumerate]{label=(\alph*)}

\usetikzlibrary{graphs,decorations.pathmorphing,decorations.markings}
\tikzcdset{scale cd/.style={every label/.append style={scale=#1},
        cells={nodes={scale=#1}}}}

\newcommand{\Z}{\mathbb{Z}}
\newcommand{\R}{\mathbb{R}}

\newcommand{\Q}{\mathbb{Q}}

\newcommand{\F}{\mathbb{F}}

\newcommand{\Hom}{\operatorname{Hom}}
\newcommand{\Aut}{\operatorname{Aut}}

\newcommand{\bigO}{\mathcal{O}}

\newcommand{\tr}{\operatorname{tr}}
\newcommand{\rk}{\operatorname{rk}}

\newcommand{\inv}{^{-1}}

\newcommand{\Res}{\operatorname{Res}}

\newcommand{\Inf}{\operatorname{Inf}}

\newcommand{\im}{\operatorname{im}}

\newcommand{\Syl}{\operatorname{Syl}}
\newcommand{\GL}{\operatorname{GL}}
\newcommand{\triv}{\mathbf{triv}}

\newcommand{\calB}{\mathcal{B}}
\newcommand{\calC}{\mathcal{C}}

\newcommand{\calE}{\mathcal{E}}
\newcommand{\calF}{\mathcal{F}}

\newcommand{\calH}{\mathcal{H}}

\newcommand{\calR}{\mathcal{R}}

\newcommand{\calT}{\mathcal{T}}

\newtheorem{theorem}{Theorem}[section]
\newtheorem{lemma}[theorem]{Lemma}
\newtheorem{prop}[theorem]{Proposition}
\newtheorem{corollary}[theorem]{Corollary}

\newtheorem*{theorem*}{Theorem}

\theoremstyle{remark}
\newtheorem{remark}[theorem]{Remark}
\newtheorem{example}[theorem]{Example}

\newtheorem{observation}[theorem]{Observation}

\theoremstyle{definition}
\newtheorem{definition}[theorem]{Definition}
\newtheorem{notation}[theorem]{Notation}
\newtheorem{construction}[theorem]{Construction}

\begin{document}
    \title{The Euler characteristic of an endotrivial complex}
\author{Nadia Mazza}
\address
{School of Mathematical Sciences\\ University of
  Lancaster\\ Lancaster \\LA1 4YF, UK}
\email{n.mazza@lancaster.ac.uk}
\author{Sam K. Miller}
\address{Department of Mathematics\\ University of Georgia \\
Athens\\ GA~30602, USA}
\email{Sam.Miller@uga.edu}
    \begin{abstract}
        Let $G$ be a finite group and $k$ a field of prime characteristic $p$. We examine the Lefschetz homomorphism $\Lambda: \calE_k(G) \to O(T(kG))$ from the group of endotrivial complexes, i.e. the Picard group of the bounded homotopy category of $p$-permutation modules $K^b({}_{kG}\triv)$, to the orthogonal unit group of the Grothendieck group of $K^b({}_{kG}\triv)$, i.e. the trivial source ring. When $p = 2$ and $k = \F_2$, $\Lambda$ is surjective when $G$ has a Sylow $2$-subgroup with fusion controlled by its normalizer, and when $G$ has dihedral Sylow $2$-subgroups. When $p$ is odd, $\Lambda$ is surjective if $G$ has a cyclic Sylow $p$-subgroup or is $p$-nilpotent, but we exhibit examples of groups of $p$-rank 2 or greater for which $\Lambda$ is not surjective. We also examine the kernel of the Lefschetz homomorphism, determining it for all groups when $p = 2$ and for groups with cyclic Sylow $p$-subgroups when $p$ is odd.
    \end{abstract}

    \maketitle

    \section{Introduction}

    Let $G$ be a finite group and $k$ a field of prime characteristic $p$. \textit{Permutation $kG$-modules}, i.e. $kG$-modules which admit a $G$-invariant basis, and their direct summands, \textit{$p$-permutation $kG$-modules}, are perhaps the most straightforward class of representations to construct. Yet they play an important role in the modular representation theory of finite groups, especially in ``local-to-global'' representation theory of $kG$, which relates the representation theory of $G$ with its \textit{$p$-local subgroups}, subgroups of $G$ of the form $N_G(P)$ with $P$ a $p$-subgroup of $G$. For instance, Brou\'e's abelian defect group conjecture predicts certain equivalences, \textit{splendid Rickard equivalences}, between homotopy categories of $p$-permutation modules. Weaker variants of the conjecture instead predict certain isomorphisms, \textit{$p$-permutation equivalences} between the corresponding Grothendieck groups, or character-theoretic equivalences known as \textit{isotypies}.

    Recently, Balmer and Gallauer \cite{BG25} deduced the Balmer spectrum of the tensor-triangulated category $K^b({}_{kG}\triv)$, the \textit{bounded homotopy category of $p$-permutation $kG$-modules}, (also denoted $\text{K}_b(p\operatorname{-perm}(G;k))$ in the literature). This category is the compact part of a ``big'' compactly generated tensor-triangulated category, the \textit{derived category of permutation modules} $\operatorname{DPerm}(G;k)$. This big category is in fact equivalent to numerous categories beyond the realm of representation theory - for instance it is equivalent to the category of cohomological Mackey functors for $G$ over $k$, the homotopy category of modules over the constant Green functor $\text{H}\underline{k}$ in genuine $G$-spectra (which was recently lifted to an $\infty$-categorical level by Fuhrmann \cite{F25}), and the category of $k$-linear Artin motives generated by motives of intermediate fields in any Galois extension with Galois group $G$. Some of these connections are reviewed in detail in \cite{BG23}, though neither these equivalences nor tensor-triangular geometry will be dwelt upon in this paper.

    Given a rigid monoidal category $\mathcal{C}$, determining its Picard group $\operatorname{Pic}(\calC)$, i.e. the group of invertible objects of $\calC$, those objects $X$ in $\mathcal{C}$ for which $X \otimes X^* \cong \mathbbm{1}_\mathcal{C}$, where $\mathbbm{1}_\mathcal{C}$ denotes the tensor unit of $\mathcal{C}$, is a pertinent question. For the category $\mathcal{C} := K^b({}_{kG}\triv)$, these invertible objects are called \textit{endotrivial complexes}, nomenclature inspired by the invertible objects the stable module category ${}_{kG}\underline{\textbf{mod}}$, endotrivial \textit{modules}. In \cite{M24c}, the second author completely determined the Picard group $\calE_k(G) := \operatorname{Pic}(K^b({}_{kG}\triv))$ for all finite groups $G$, identifying for $p$-groups endotrivial complexes with $G$-representation spheres and \textit{Borel-Smith functions}.

    Another object of interest is the \textit{trivial source ring} $T(kG)$, the Grothendieck group of $K^b({}_{kG}\triv)$ as a tensor-triangulated category, and equivalently, the split Grothendieck group of the additive category ${}_{kG}\triv$. Boltje and Carman in \cite{BC23} gave a character-theoretic description of $T(kG)$ in the language of \textit{coherent character tuples}. In particular, they gave a characterization of the \textit{orthogonal unit group} $O(T(kG))$, the group of units $u \in T(kG)^\times$ satisfying $u\inv = u^*$. These orthogonal units induce $p$-permutation equivalences \cite{BP20} in the same manner that endotrivial complexes induce splendid Rickard equivalences \cite{M24a}.

    We have a natural group homomorphism $\Lambda: \calE_k(G) \to O(T(kG))$ induced by taking the \textit{Euler characteristic} (also called the \textit{Lefschetz invariant}) of an endotrivial complex $C$, i.e. its image in $T(kG)$, an alternating sum of terms. This map $\Lambda$ is called the \textit{Lefschetz homomorphism}, and is the central focus of this paper. In \cite{M24c}, it was shown that when $G$ is a $p$-group, $\Lambda$ is surjective. On the other hand, for non-$p$-groups, if $k$ contains more than $p$ elements, $\Lambda$ may not be surjective \cite{M24a}.

    In this paper, we investigate further the Lefschetz homomorphism for non-$p$-groups. The cases of $p = 2$ and $p$ odd play out rather differently (as usual), due to the characterization of $O(T(kG))$ by Boltje and Carman. For $p = 2$, the question of surjectivity is inherently tied to the unit group of the Burnside ring $B(G)^\times$ of a finite group $G$, whereas for $p$ odd, the question instead is tied to Boltje and Carman's coherent character tuples (in part because for $p$-groups $G$ with $p$ odd, $B(G)^\times$ has order 2). Though a complete characterization of the kernel and image of $\Lambda$ remains a difficult question, we make significant progress on the question, which we summarize now. These results arise in Section \ref{sec:p2}.

    \begin{theorem*}
        Set $k = \F_2$. Let $S$ be a Sylow $2$-subgroup of $G$. Then $\Lambda: \calE_k(G) \to O(T(kG))$ is surjective if the normalizer of $S$ controls the $2$-fusion in $G$, or if $S$ is a dihedral $2$-group.
    \end{theorem*}

    In particular, $\Lambda$ is surjective over $\F_2$ for any group $G$ with abelian Sylow 2-subgroups, or more generally, \textit{resistant} Sylow 2-subgroups. When $k$ has characteristic 2 and has more than 2 elements, $\Lambda$ is usually non-surjective. On the other hand, we find no counterexamples of $\Lambda$ not surjective when $k = \F_2$. We pose the question of whether $\Lambda: \calE_{\F_2}(G) \to O(T(\F_2 G))$ is always surjective.

    The situation for $p$ odd is rather different. We derive the following results in Section \ref{sec:podd}.

    \begin{theorem*}
        Let $p$ be an odd prime and $k$ be a field contained in $\overline{\F}_p$. If $G$ is $p$-nilpotent or has a cyclic Sylow $p$-subgroup, then $\Lambda: \calE_k(G) \to O(T(kG))$ is surjective. On the other hand, for every odd prime $p$ and integer $i \geq 2$, there exist finite groups with elementary abelian Sylow $p$-subgroups of rank $i$ such that $\Lambda: \calE_k(G) \to O(T(kG))$ is not surjective.
    \end{theorem*}

    An exact criterion for when $\Lambda$ is non-surjective for $p$ odd appears to be unfeasible, but our example demonstrates the expectation that for groups $G$ with Sylow $p$-subgroups having sufficiently many generators and whose fusion in $G$ is sufficiently complex, $\Lambda$ should not be surjective.

    Finally, we consider the kernel of the Lefschetz homomorphism in Section \ref{sec:kernel}. For $p = 2$, the kernel of $\Lambda$ is easily determined, and for any prime $p$, when $G$ has a cyclic Sylow $p$-subgroup, we compute the kernel directly. For $p$ odd, we determine a necessary condition for an endotrivial complex $C$ to belong to the kernel of the Lefschetz homomorphism, but an exact characterization of the kernel again remains out of reach. We note that Bachmann in his Ph.D. thesis \cite{Bac16} performed similar computations with more restrictive settings and in the category of Artin motives.

    We note a potential application of these computations. Two related questions of interest are whether a $p$-permutation equivalence implies a splendid Rickard equivalence, and if, given a $p$-permutation equivalence, it is induced by a splendid Rickard equivalence. At present, no examples of two block algebras which are $p$-permutation equivalent but not splendid Rickard equivalent are known. Similarly, no examples of $p$-permutation equivalences not induced by splendid Rickard equivalences are known either. In fact, the second author showed in \cite{M24c} that every $p$-permutation autoequivalence of the group algebra $kP$, with $P$ a $p$-group, is induced by a splendid Rickard autoequivalence. Since endotrivial complexes induce splendid Rickard equivalences and orthogonal units of the trivial source ring induce $p$-permutation equivalences, we hope the results of this paper can be used to construct examples of $p$-permutation equivalences not induced from splendid Rickard equivalences.

    \textbf{Acknowledgments:} The second author thanks Lancaster University for their hospitality during the visit in which this paper was initiated, as well as UC Santa Cruz for their generous travel support which made the visit possible. He also thanks Jesper Grodal for some very helpful conversations illuminating the arguments used in \cite{Tor84}, and Robert Boltje for his advice and suggestions leading to the topic of this paper.

    \section{Preliminaries on endotrivial complexes}

    Throughout this paper, all modules are assumed finitely generated. Recall that a \textit{permutation $kG$-module} is a $kG$-module with a $G$-stable basis, and a \textit{$p$-permutation $kG$-module} is a direct summand of a permutation module. A $kG$-module $M$ is a $p$-permutation module if and only if $\Res^G_S M$ is a permutation $kS$-module, where $S$ denotes a Sylow $p$-subgroup of $G$. Indecomposable $p$-permutation modules are often referred to as \textit{trivial source modules}, as an indecomposable $kG$-module is $p$-permutation if and only if it has trivial source.

    Given any $p$-subgroup $P$ of $G$, we have a monoidal functor $-(P): {}_{kG}\triv\to {}_{k[N_G(P)/P]}\triv$, the \textit{Brauer construction} (also referred to as the \textit{Brauer quotient}, and in \cite{BG25}, this is naturally isomorphic to the \textit{modular fixed points} functor). The Brauer construction satisfies the property that for any $G$-set $X$, one has a natural isomorphism of $k[N_G(P)/P]$-modules $k[X](P) \cong k[X^P]$. For in-depth overviews of $p$-permutation modules or the Brauer construction, we refer the reader to \cite[Chapter 5]{L181} or \cite{CL23}.

    \begin{definition}
        Let $C$ be a bounded chain complex of $p$-permutation $kG$-modules. We say $C$ is \textit{endotrivial} if \[C^* \otimes_k C \simeq k[0],\] where $\simeq$ denotes homotopy equivalence of chain complexes. In other words, $C$ is an invertible object in the bounded homotopy category of $p$-permutation $kG$-modules, $K^b({}_{kG}\triv).$ We let $\calE_k(G)$ denote the Picard group of $K^b({}_{kG}\triv)$. In other words, $\calE_k(G)$ denotes the set of isomorphism classes of endotrivial complexes in $K^b({}_{kG}\triv).$ Then $\calE_k(G)$ is an abelian group with group addition induced by $[C] + [D] = [C \otimes_k D]$ and inverse given by $-[C] = [C^*]$. We may abusively write $C \in \calE_k(G)$, in which case we assume $C$ is the unique indecomposable representative of the equivalence class $[C] \in \calE_k(G)$. Indeed, any endotrivial complex has a unique noncontractible direct summand.
    \end{definition}

    \begin{construction}
        Let $C$ be an endotrivial complex of $kG$-modules. Observe that for any $p$-subgroup $P$ of $G$, $C(P)$ is an endotrivial complex of $p$-permutation $k[N_G(P)/P]$-modules. Additionally, it follows from the K\"{u}nneth formula that any endotrivial complex has nonzero homology concentrated in a unique degree, with that homology having $k$-dimension 1. Given any $p$-subgroup $P$ of $G$, let $h_C(P)$ denote the unique integer for which $H_{h_C(P)}(C(P))\neq 0$.

        We denote the set of $p$-subgroups of $G$ by $s_p(G)$. Then $h_C: s_p(G) \to \Z$ is a superclass function (i.e. a function from a subset of subgroups of $G$ to $\Z$ that agrees on conjugates) on the $p$-subgroups of $G$. We call $h_C$ the \textit{h-mark function of $C$}. We denote the additive group of superclass functions on $p$-subgroups of $G$ by $CF(G,p)$.

        Since the nonzero homology of $C$ has $k$-dimension 1, it is isomorphic to a $kG$-module of the form $k_\omega$, for some $\omega \in \Hom(G, k^\times)$. Given any $p$-subgroup $P$ of $G$, let $\calH_C(P)$ denote the unique homomorphism $\omega \in \Hom(N_G(P)/P, k^\times)$ such that $k_\omega \cong H_{h_C(P)}(C(P))$.
    \end{construction}

    \begin{prop}
        \begin{enumerate}
            \item The assignment $C \mapsto h_C$ induces a well-defined group homomorphism $h: \calE_k(G) \to CF(G,p)$.
            \item The assignment $C \mapsto H_{h_C(P)}(C(P))$ induces a well-defined group homomorphism $\calH_C: \calE_k(G) \mapsto \Hom(N_G(P)/P,k^\times)$. This further induces a well-defined group homomorphism:
            \begin{align*}
                \calH: \calE_k(G) &\to\left( \prod_{P \in s_p(G)} \Hom(N_G(P)/P, k^\times)\right)^G\\
                C &\mapsto \{\calH_C(P)\}_{P \in s_p(G)}
            \end{align*}
        \end{enumerate}

    \end{prop}

    \subsection{Borel-Smith functions}

    The classification of endotrivial complexes relies heavily on Borel-Smith functions, which we now introduce. We note that we adopt the convention of only considering Borel-Smith functions valued on $p$-subgroups of a group, but Borel-Smith functions are also regarded superclass functions over all subgroups in the literature.

    \begin{definition}
        A \textit{Borel-Smith} function is a superclass function $f \in CF(G,p)$ satisfying the following three conditions, which we call the Borel-Smith conditions.
        \begin{enumerate}
            \item If $p$ is odd, then for any subquotient $T/S$ of $G$ of order $p$, $f(T) \equiv f(S) \mod 2$.
            \item If $p = 2$, then for any sequence of subgroups $H \trianglelefteq K \trianglelefteq L \leq N_G(H)$, with $[K:H] = 2$, $f(K) \equiv f(H) \mod 2$ if $L/K$ is cyclic of order 4 and $f(K) \equiv f(H) \mod 4$ is $L/K$ is quaternion of order 8.
            \item For any elementary abelian subquotient $T/S$ of $G$ of rank 2, the equality \[f(S) - f(T) = \sum_{S < X < T} \big(f(X) - f(T)\big)\] holds.
        \end{enumerate}
        The collection of Borel-Smith functions forms an additive subgroup $CF_b(G,p)$ of $CF(G,p)$. In fact, if $G$ is a $p$-group, then under the identification $CF(G) \cong B^*(G) := \Hom(B(G), \Z)$, $CF_b(G)$ forms a rational $p$-biset subfunctor of $CF(G)$. See \cite[Proposition 3.7]{BoYa07} for details. We refer the reader to \cite{Bou10} for an overview of biset functors, although a background understanding of biset functors is unnecessary for this paper.

    \end{definition}

    \begin{remark}
        Let $f\in CF_b(G,p)$ be a Borel-Smith function for $G$. Observe that if $P$ is a non-cyclic $p$-subgroup of $G$ (with the convention that the trivial subgroup is cyclic), that the value $f(P)$ is determined entirely by subgroups $Q < P$. Indeed, since $P$ is non-cyclic, there exists a normal subgroup $R \trianglelefteq P$ such that $P/R$ is elementary abelian of rank 2. Therefore, the $\Z$-rank of $CF_b(G,p)$ is bounded by the number of conjugacy classes of cyclic $p$-subgroups of $G$.
    \end{remark}

    In fact, for $p$-groups, this bound is known to be an equality.

    \begin{prop}{\cite[Proposition III.5.10, page 214]{tD87}}
        Let $G$ be a $p$-group. The $\Z$-rank of $CF_b(G)$ is exactly the number of conjugacy classes of cyclic subgroups of $G$.
    \end{prop}

    \begin{remark}
        Borel-Smith functions were originally considered in the context of homotopy representations of spheres, see \cite[Section III.5: Borel-Smith functions]{tD87} for details. Purely representation-theoretically, Borel-Smith functions arise from real representations (or equivalently, representation spheres) in the following manner. Given a field $\F$ and a $\F G$-module $V$, define the superclass function $\dim(V)$ as follows: \[\dim(V)(H) := \dim_\F V^H.\] If $\F G$ is semisimple, then \[\dim(V)(H) = \frac{1}{|H|} \sum_{h \in H} \chi_V(h),\] where $\chi_V$ is the character of $V$. We call $\dim$ the \textit{dimension homomorphism}. In this case, $\dim: R_\F(G) \to CF(G)$ is a group homomorphism. Regarding a representation $V$ as a representation sphere $S^V$ via one-point compactification, the dimension homomorphism $\dim(V)(H)$ simply returns the dimension of the representation sphere $S^{V^H}$.

        If $G$ is a $p$-group and $\F = \R$, then the image of $\dim: R_\R(G) \to CF(G)$ is exactly $CF_b(G)$, the group of Borel-Smith functions. In fact, $\dim: R_\R \to CF_b$ is a morphism of rational $p$-biset functors.
    \end{remark}

    In fact, one can determine precisely when a Borel-Smith function arises from a real representation.

    \begin{prop}{\cite{tD87}}
        Let $f \in CF_b(G)$ be a Borel-Smith function. Then $f$ is monotonically decreasing (with respect to the poset of subgroups of $G$) and non-negative if and only if $f = \dim_\R V$ for a real representation $V$.
    \end{prop}

    \begin{prop}
        Let $G$ be a $p$-group and let $V_1, \dots, V_n$ denote the irreducible real representations of $G$. Then the set of associated dimension functions $\{f_1,\dots, f_n\}$, after removing duplicates, forms a $\Z$-basis of $CF_b(G)$.
    \end{prop}
    \begin{proof}
        Since the image of the dimension homomorphism is precisely $CF_b(G)$, it suffices to show that the set $\calB := \{f_1, \dots, f_n\}$ is linearly independent after removing duplicates.  \cite[Proposition III.5.9, page 213]{tD87} asserts that $\ker(\dim)$ is generated by elements of the form $V - \psi^k(V)$, where $V$ is an irreducible real representation, $\psi^k$ is the $k$-th Adams operation, and $k$ is coprime to $|G|$. Therefore, the duplicates in $\calB$ arise from Adams operation conjugates, and it follows that any set of real irreducible representations which are not Adams operation conjugates will correspond to a linearly independent set of Borel-Smith functions.
    \end{proof}

    We next show that even if $G$ is not a $p$-group, $CF_b(G,p)$ has $\Z$-rank the number of conjugacy classes of $p$-subgroups of $G$. Given a subgroup $H$ of $G$, say $f \in CF_b(H,p)$ is \textit{$G$-stable} if, for all pairs of $G$-conjugate $p$-subgroups $Q_1, Q_2$ of $H$, $f(Q_1) = f(Q_2)$. Let $CF_b(H,p)^G$ be the subgroup of $CF_b(H,p)$ consisting of $G$-stable elements. Note that given a Sylow $p$-subgroup $S$ of $G$, we have an identification $CF_b(S)^G = CF_b(G,p)$, which we use without further mention.

    \begin{lemma}\label{lem:gstablefromcyclics}
        Let $H$ be a subgroup of $G$ and let $f \in CF_b(H,p)$. If $f$ is $G$-stable at all cyclic $p$-subgroups of $H$, then $f$ is $G$-stable.
    \end{lemma}
    \begin{proof}
        We prove this by induction on the order of $p$-subgroups of $H$. This is true by assumption on all cyclic $p$-subgroups, in particular, it is true for all subgroups of order 1 and $p$. Now suppose $Q_1, Q_2$ are two $p$-subgroups of $H$ of order at least $p^2$ that are $G$-conjugate. If $Q_1, Q_2$ are cyclic, then $f(Q_1) = f(Q_2)$ by assumption. Otherwise, both are non-cyclic, and therefore have subgroups $R_1$ of $Q_1$ and $R_2$ of $Q_2$ of index $p^2$ that are $G$-conjugate and for which $Q_i/R_i \cong C_p \times C_p$ for $i \in \{1, 2\}$. It follows by induction that since $f(Q_i)$ is uniquely determined by the values $f(R_i)$ and $f(X_i)$ for all intermediate subgroups $R_i < X_i < Q_i$ for $i \in \{1, 2\}$, $f(Q_1) = f(Q_2)$, as desired.
    \end{proof}

    \begin{lemma}\label{lem:diagonalbasis}
        Let $G$ be a $p$-group, and let $\{H_1,\dots, H_r\}$ be a collection of representatives of all conjugacy classes of cyclic subgroups of $G$. Then $CF_b(G)$ has a linearly independent set of $r$ elements $\{f_1, \dots, f_r\}$ such that the following holds; for all $i\neq j \in \{1, \dots, r\}$, we have $f_i(H_i) \neq 0$ and $f_i(H_j) = 0$.
    \end{lemma}
    \begin{proof}
        To prove this, it suffices to show that the $\Q$-vector space $\Q CF_b(P) := \Q \otimes_\Z CF_b(P)$ has a basis satisfying the same property (note this is basis since $CF_b(P)$ has rank equal to $r$). Let $c(P)$ denote the set of cyclic subgroups of $P$, and let $CF_b(P,c)$ denote the set of functions $c(P) \to \Z$ satisfying the Borel-Smith conditions. We have an obvious projection map $\pi: \Q CF_b(P) \to \Q CF_b(P,c)$.

        We claim $\pi$ is an isomorphism. First, we show $\pi$ is injective. Indeed, any element $f \in \ker (\pi)$ satisfies $f(C) = 0$ for any cyclic subgroup $C$ of $P$. However, since the Borel-Smith conditions completely determine the value of $f$ on non-cyclic subgroups of $P$, it follows that $f = 0$. Thus $\pi$ is injective. On the other hand, $CF_b(P,c)$ has rank at most $r$ by construction, so $\pi$ is surjective as well, and thus an isomorphism. Taking the inverse image of the standard basis of $\Q CF_b(P,c)$ via $\pi$ gives a basis of $\Q CF_b(P)$ which satisfies the desired property, and we are done.
    \end{proof}

    Denote the number of conjugacy classes of cyclic $p$-subgroups of a finite group $G$ by $c(G)$.

    \begin{theorem}
        Let $H$ be a subgroup of $G$, let $\{K_1, \dots, K_r\}$ be a collection of representatives of the $G$-conjugacy classes of cyclic $p$-subgroups of $H$, and let $\{K_1,\dots, K_r, \dots, K_s\}$ be a collection of representatives of the $H$-conjugacy classes of cyclic $p$-subgroups of $H$. Then, the following hold.
        \begin{enumerate}
            \item $CF_b(H,p)^G$ has a linearly independent set of $r$ elements $\{f_1,\dots, f_r\}$ such that the following holds; for all $i \neq j \in \{1,\dots, r\}$, we have $f_i(K_i) \neq 0$ and $f_i(K_j) = 0$.
            \item In particular, if $H$ contains a Sylow $p$-subgroup of $G$, then $CF_b(H,p)^G$ has $\Z$-rank $c(G)$.
        \end{enumerate}
    \end{theorem}
    \begin{proof}
        (b) follows from (a) since if (a) holds, $CF_b(H,p)^G$ has rank at least $r$, but $CF_b(H,p)^G$ must have rank at most $r$ since Borel-Smith functions are determined by their values at cyclic $p$-subgroups. To show (a), it again suffices to show $\Q CF_b(H,p)^G := \Q \otimes_\Z CF_b(H,p)^G$ has a basis satisfying the same property. Lemma \ref{lem:diagonalbasis} asserts $\Q CF_b(H,p)$ has a basis consisting of elements $\{g_1,\dots, g_s\}$ such that $g_i(K_i) = 1$ and $g_i(K_j) = 0$ when $i \neq j$. For $i \in \{1, \dots, r\}$, we define $f_i \in CF_b(H,p)^G$ as follows:
        \[f_i = \sum_{j = 1}^s \gamma^G(i,j) g_j, \text{  with  } \gamma^G(i,j) = \begin{cases} 1 & K_i, K_j \text{ are $G$-conjugate}\\ 0 & \text{otherwise.} \end{cases}\]
        Since $K_1,\dots, K_r$ all lie in distinct conjugacy classes, $\{f_1,\dots f_r\}$ is a linearly independent set. By construction, all $f_i$ are $G$-stable on the cyclic $p$-subgroups of $H$, therefore Lemma \ref{lem:gstablefromcyclics} implies that all $f_i$ are $G$-stable. Since $\Q CF_b(H,p)^G$ has dimension at most $r$, $\{f_1,\dots, f_r\}$ is a basis of $\Q CF_b(H,p)^G$, as desired.
    \end{proof}

    \subsection{The classification of endotrivial complexes}

    \begin{theorem}{\cite{M24c}}
        Let $G$ be a finite group. The following statements hold.
        \begin{enumerate}
            \item $\ker (h): \calE_k(G) \to CF(G,p)$ is equal to the torsion subgroup of $\calE_k(G)$, which is identified to the group $\Hom(G,k^\times)$ by the assignment $\omega \mapsto (k_\omega)[0]$.
            \item $\im (h) = CF_b(G,p)$. In particular, $\calE_k(G)$ has $\Z$-rank equal to the number of conjugacy classes of cyclic $p$-subgroups of $G$.
        \end{enumerate}
        If $G$ is a $p$-group, then we have an isomorphism $\calE_k(G) \cong CF_b(G)$, and $\calE_k$ is a rational $p$-biset functor via the transport of structure $\calE_k \cong CF_b$.
    \end{theorem}

    \begin{corollary}\label{cor:etrivclassification}
        We have a split exact sequence of abelian groups \[0 \to \Hom(G,k^\times) \to \calE_k(G) \to CF_b(G,p) \to 0.\] In particular, $\rk_\Z \calE_k(G) = c(G)$, which is independent of the field extension  $k/\F_p$.
    \end{corollary}

    That the rank of $\calE_k(G)$ is independent of the size of the field extension $k/\F_p$ suggests that there is a certain free subgroup of $\calE_k(G)$ which is realized over the finite field $\F_p$. Indeed, this is the case. Let $k'/k$ be a field extension.  We say that a chain complex of $k'G$-modules $C'$ \textit{descends to $k$} if there exists a chain complex of $kG$-modules $C$ for which $k' \otimes_k C \cong C'$.

    \begin{theorem}
        Let $C$ be an indecomposable endotrivial complex of $k'G$-modules. If $\calH_C(1) = \omega \in \Hom(G,k^\times) \leq \Hom(G, k'^\times)$ for some subfield $k$ of $k'$, then $C$ descends to $k$.
    \end{theorem}
    \begin{proof}
        This is a straightforward consequence of \cite[Corollary 4.4]{M24e}, since any endotrivial complex contains exactly one noncontractible summand.
    \end{proof}

    \begin{example}
        The following example will be relevant in the sequel. Let $G$ be a cyclic $p$-group of order $p^n$, and for $i \in \{0,\dots, n\}$, let $H_n$ denote the unique subgroup of $G$ of order $p^i$. Then $\calE_k(G)$ has a $\Z$-basis $\{C_0,\dots, C_n\}$ with $C_n = k[1]$ and for $i < n - 1$, $C_i = k[G/H_i] \to k[G/H_i] \to k$, with the trivial $kG$-module in degree zero. If $p = 2$, $C_{n-1} = k[G/H_{n-1}] \to k$ and if $p > 2$, then $C_{n-1} = k[G/H_{n-1}] \to k[G/H_{n-1}] \to k$. Each $C_i$ corresponds to the unique irreducible representation of $\R G$ with kernel $H_i$.
    \end{example}

   There are two free subgroups of $\calE_k(G)$ of note, both of which consist entirely of elements with representatives which descend to $\F_p$.

   \begin{definition}
       \begin{enumerate}
           \item Let $\calT\calE_k(G)$ denote the subgroup of $\calE_k(G)$ consisting of endotrivial complexes $C$ satisfying $\calH_C(1) = 1 \in \Hom(G,k^\times)$. Any (unique) indecomposable representative of an element in $\calE_k(G)$ descends to $\F_p$, so we may simply write $\calT\calE(G)$ without confusion. We have an obvious restriction map for any subgroup $H$ of $G$, $\calT\calE(G) \to \calT\calE(H)$, but the Brauer construction does \textit{not} induce a group homomorphism $\calT\calE(G) \to \calT\calE(N_G(P)/P)$ in general for $p$-subgroups $P$ of $G$. Note we have an isomorphism $\calE_k(G) \cong \Hom(G,k^\times) \times \calT\calE(G),$ see \cite[Corollary 4.6]{M24e}.

           \item The following construction was first considered in \cite{M24b} and used in \cite{M24c}. Define the $kG$-module $V(\calF_G) := \bigoplus_{P \in s_p(G)\setminus \Syl_p(G)} k[G/P]$. We say a chain complex $C$ is \textit{relatively $V(\calF_G)$-endotrivial} if $C^* \otimes_k C \simeq (k \oplus M)[0]$ for some $V(\calF_G)$-projective $kG$-module $M$. Note this follows the language convention of \cite{M24c}. Explicitly this means $M$ has no indecomposable direct summands with Sylow vertex.

           One has a corresponding group parametrizing the relatively $V(\calF_G)$-endotrivial complexes, $\calE_k^{V(\calF_G)}(G)$. Each element of $\calE_k^{V(\calF_G)}(G)$ has a unique indecomposable representative that is relatively $V(\calF_G)$-endotrivial. It is easy to verify that for any $p$-subgroup $P$ of $G$ and any relatively $V(\calF_G)$-endotrivial complex $C$, $C(P)$ has nonzero homology in exactly one degree. Therefore, we have a corresponding notion of h-marks and a h-mark homomorphism $h: \calE_k^{V(\calF_G)}(G) \to CF(G,p)$. In fact, this homomorphism is surjective by \cite[Theorem 2.12]{M24c}, and moreover, we have a free subgroup ${}^\Omega\calE_k^{V(\calF_G)}(G)$ of maximal rank spanned by the two-term complexes $C_P := k[G/P] \to k$, where $P$ is a non-Sylow $p$-subgroup of $G$. We call these chain complexes \textit{relative syzygies}.

           Now, we have an obvious inclusion $\calE_k(G) \to \calE_k^{V(\calF_G)}(G)$, therefore $\calE_k(G)$ may be regarded as a subgroup of $\calE_k^{V(\calF_G)}(G)$. Explicitly, it is the subgroup of $\calE_k^{V(\calF_G)}(G)$ consisting of elements whose unique indecomposable representative is endotrivial. Finally, we set $\calE_k^\Omega(G):= \calE_k(G) \cap {}^\Omega\calE_k^{V(\calF_G)}(G)$, the subgroup of $\calE_k^\Omega(G)$ consisting of endotrivial complexes that arise as direct summands of tensor products of relative syzygies. We have an isomorphism $\calE_k^\Omega(G) \cong CF_b(G,p)$.

           It is easy to see that for any element $C \in \calE_k^\Omega(G)$, $\calH_C(S) = 1$ for any Sylow $p$-subgroup $S$. Conversely, determining even $\calH_C(1)$ is computationally quite difficult.
       \end{enumerate}
    \end{definition}

    \begin{remark}
        If $G$ is a $p$-group, $\calE_k(G) \cong CF_b(G)$; morally, all endotrivial complexes arise from representation spheres. A more direct topological argument of this fact is given in \cite[Theorem 7.8]{F25}. In particular, every monotonically decreasing endotrivial complex is realized as the singular (Bredon) cohomology (see e.g. \cite{I73}) of a representation sphere, since every representation sphere has the structure of a $G$-CW-complex. Moreover, every monotonically decreasing endotrivial complex over $G$ arises as a chain complex of $k\Gamma_G$-modules, where $\Gamma_G$ denotes the \textit{orbit category} of $G$. (see e.g. \cite{Y16}).

        It is verified for $p$-groups in \cite{M24c} that restriction and inflation of Borel-Smith functions correspond to restriction of endotrivial complexes, and deflation of Borel-Smith functions corresponds to the Brauer construction. Induction of Borel-Smith functions corresponds to taking traces of representation spheres, which induce endotrivial complexes via Bredon cohomology.
    \end{remark}

    \section{Orthogonal units of the trivial source ring}

    \begin{definition}
        The \textit{trivial source ring} $T(kG)$ is the (split) Grothendieck ring of ${}_{kG}\triv$. Given a $p$-permutation $kG$-module $M$, let $[M]$ denote its image in $T(kG).$ It has, as an abelian group, a free $\Z$-basis spanned by isomorphism classes of indecomposable $p$-permutation $kG$-modules (also called trivial source modules), with addition induced by direct sums: \[[M] + [N] := [M \oplus N].\] Multiplication is induced by the tensor product: \[[M]\cdot [N] := [M \otimes_k N].\]

        The assignment $[M] \mapsto [\Hom_k(M,k)]$ induces a ring involution on $T(kG)$, which we also denote by $(-)^*$. The \textit{orthogonal unit group} of $T(kG)$, denoted $O(T(kG))$, is the subgroup of the unit group $T(kG)^\times$ given by all units $u \in T(kG)^\times$ for which $u\inv = u^*$. Equivalently, $O(T(kG))$ is the torsion subgroup of $T(kG)^\times$, see \cite[Remark 4.1(b)]{BC23}.

        Equivalently, $T(kG)$ is the Grothendieck ring of the tensor-triangulated category $K^b({}_{kG}\triv)$. In this case, the image of any chain complex $C \in K^b({}_{kG}\triv)$, denoted $\Lambda(C)$, in $T(kG)$ is called the \textit{Lefschetz invariant} or \textit{Euler characteristic} of $C$. Explicitly, we have \[\Lambda(C) = \sum_{i \in \Z} (-1)^i[C_i] \in T(kG).\]
    \end{definition}

    \begin{prop}{\cite[Proposition 4.1]{M24a}}
        Let $C$ be an endotrivial complex of $kG$-modules. Then $\Lambda(C)$ is an orthogonal unit of $T(kG)$. Moreover, the assignment $C \mapsto \Lambda(C)$ induces a well-defined group homomorphism \[\Lambda: \calE_k(G) \to O(T(kG)).\]
    \end{prop}

    We call $\Lambda$ the \textit{Lefschetz homomorphism}. The structure of the trivial source ring, and its orthogonal unit group, was described by Boltje and Carman in \cite{BC23}. We will summarize their results describing the structure of $O(T(kG))$. It turns out that the structure is partially governed by the \textit{Burnside ring} $B(G)$ of $G$ \cite[Section I.2]{tD87}.

    \begin{definition}
        The \textit{Burnside ring} $B(G)$ of a finite group $G$ is the Grothendieck ring of the category ${}_{G}\mathbf{set}$ of finite $G$-sets. Given a $G$-set $X$, let $[X]$ denote its image in $B(G).$ As an abelian group, $B(G)$ is free with $\Z$-basis given by isomorphism classes of indecomposable $G$-sets, which have the form $G/K$ for subgroups $K$ of $G$. Addition is induced by disjoint unions \[[X] + [Y] := [X \sqcup Y],\] and multiplication is induced by the direct product \[[X]\cdot [Y] := [X\times Y].\]

        Given a $G$-set $X$ and a subgroup $H$ of $G$, the value $|X^H|$ is referred to as the \textit{mark of $X$ at $H$}. Given a $G$-set $X$, we obtain a superclass function $m_X$ given by \[m_X(H) = |X^H|.\] This assignment induces a well-defined group homomorphism $m: B(G) \to CF(G)$, which is referred to as the \textit{mark homomorphism.} This homomorphism is well-known to be injective and full-rank, but not surjective in general.

        The unit group of the Burnside ring, $B(G)^\times$, has image contained in $CF(G)^\times$, which has order $2^{|s(G)|}$, where $s(G)$ denotes a collection of representatives of conjugacy classes of subgroups of $G$. In particular, $B(G)^\times$ is finite dimensional $\F_2$-vector space. The structure of $B(G)^\times$ is famously hard to describe. In fact, tom Dieck proved the following: ``if $G$ has odd order, then $B(G)^\times = \{\pm [G/G]\}$'' can be shown to be equivalent to the Feit-Thompson odd order theorem, see \cite[Theorem 11.2.4]{Bou10}. However, Yal\c{c}in in \cite{Y05} gave a list of generators for $B(P)^\times$ for any $2$-group $P$, which Bouc in \cite{Bo07} refined to a $\F_2$-basis using his theory of biset functors and genetic subgroups. See \cite{Bou10} for a detailed treatment.

        Let $S$ be a Sylow $p$-subgroup of $G$. The \textit{$G$-stable subgroup of $B(S)$}, denoted $B(S)^G$, is the subgroup of $B(S)$ given by all elements $x$ for which $m_x(H_1) = m_x(H_2)$ for all pairs of $G$-conjugate subgroups $H_1, H_2$ of $S$.
    \end{definition}

    \begin{theorem}{\cite[Theorem C]{BC23}}
        Let $S$ be a Sylow $p$-subgroup of $G$. We have a direct product decomposition \[O(T(kG)) \cong (B(S)^G)^\times \times \left(\prod_{P \in s_p(G)} \Hom(N_G(P)/P, k^\times)\right)'\] where the second factor is defined as the set of all $G$-stable tuples \[(\varphi_P) \in \left(\prod_{P \in s_p(G)} \Hom(N_G(P)/P, k^\times)\right)^G\] additionally satisfying \[\varphi_P(xP) = \varphi_{P\langle x_p\rangle}(xP\langle x_p \rangle),\] for all $P \in s_p(G)$ and $x \in N_G(P)$. Here, $x_p$ denotes the \textit{$p$-part of $x$}: for every $x \in G$, there exists a unique $x_p \in G$ and $x_{p'} \in G$ such that $x_p$ is a $p$-element, $x_{p'}$ is a $p'$-element, and $x = x_px_{p'} = x_{p'}x_p$.
    \end{theorem}

    \begin{remark}
        We describe this decomposition in greater detail, following \cite{BC23}. First, \cite[Theorem A]{BC23} states there is an injective ring homomorphism \[\beta_G: T(kG) \to \left(\prod_{P \in s_p(G)} R(K[N_G(P)/P])\right)^G,\] whose image consists of character tuples satisfying the same coherence condition: for each $P \in s_p(G)$ and $x \in N_G(P)$, one has $\chi_P(xP) = \chi_{P\langle x_p\rangle} (xP\langle x_p\rangle)$. In particular, \[(\beta_G(x))_P = K\otimes_\bigO \widehat{x(P)} \in R(K[N_G(P)/P]),\] where $x(P)$ denotes the Brauer construction applied to $x$, and $\widehat{(-)}$ denotes the isomorphism $T(\bigO G) \cong T(kG)$ induced by taking the unique lift of a $p$-permutation $kG$-module to a $p$-permutation $\bigO G$-module. Since the unit group of $R_K(G)$ is generated by virtual $k$-dimension one characters, for every orthogonal unit $u\in O(T(kG))$, there exist homomorphisms $\varphi_P\in \Hom(N_G(P)/P, K^\times)$ and signs $\epsilon_P \in \{\pm 1\}$ such that  \[\beta_G(u) = (\epsilon_P\cdot \varphi_P)_{P \in s_p(G)}.\] However, the coherence condition implies that if $x \in G$ is a $p$-element, $\varphi_P(x) = 1$, so each $\varphi_P$ descends to a degree one Brauer character, hence a homomorphism $\overline{\varphi_P} \in \Hom(N_G(P)/P, k^\times)$. Therefore, $\beta_G$ may be regarded as a group homomorphism \[\beta_G: O(T(kG)) \to \left(\prod_{P \in s_p(G)}\{\pm 1\} \times \Hom(N_G(P)/P, k^\times)\right)^G.\]

        One may explicitly compute for each $p$-subgroup $P$ of $G$ $\epsilon_P \cdot \varphi_P$ by taking a lift of $u(P)\in O(T(k[N_G(P)/P]))$ in $R_k(G)$ to obtain the signed degree one Brauer character.

        Finally, we have the identification \begin{align*}
            (B(S)^G)^\times \times \left(\prod_{P \in s_p(G)} \Hom(N_G(P)/P, k^\times)\right)^G &\cong \left(\prod_{P \in s_p(G)}\{\pm 1\} \times \Hom(N_G(P)/P, k^\times)\right)^G \\
            \big(x, (\varphi_P)_{P \in s_p(G)}\big) &\mapsto (m_x(P)\cdot \varphi_P)_{P\in s_p(G)}
        \end{align*}

        Altogether, we have the following isomorphism. We identify $(B(S)^G)^\times$ with its image in $CF(G,p)$, and let $\beta^\varphi_G$ denote the projection of $\beta_G$ onto its second component, the character tuple.
        \begin{align*}
            \kappa: O(T(kG)) &\cong (B(S)^G)^\times \times \left(\prod_{P \in s_p(G)} \Hom(N_G(P)/P, k^\times)\right)'\\
            u &\mapsto \big(( P \mapsto \dim_k u(P)), \beta_G^\varphi(u)\big)
        \end{align*}
    \end{remark}

    In fact, the tuple of coherent tuples has a further decomposition, which will be of use in the sequel.

    \begin{prop}{\cite[Proposition 4.6]{BC23}}\label{prop:furtherdecomp}
        One has a decomposition \[\left(\prod_{P \in s_p(G)} \Hom(N_G(P)/P,k^\times)\right)' \cong \Hom(G,k^\times) \times \left(\prod_{P \in s_p(G)} \Hom(N_G(P)/PC_G(P),k^\times)\right)',\] where the second factor denotes the set of all tuples \[(\varphi_P) \in \left(\prod_{P \in s_p(G)} \Hom(N_G(P)/PC_G(P),k^\times)\right)^G\] satisfying \[\varphi_P\big(xPC_G(P)\big) = \varphi_{P\langle x_p\rangle}\big(xP\langle x_p\rangle C_G(P\langle x_p\rangle)\big)\] for all $P \in s_p(G)$ and $x \in N_G(P)$.
    \end{prop}

    Note that if $P$ is abelian, then $PC_G(P) = C_G(P)$, hence the quotient $N_G(P)/PC_G(P)$ is equivalently the \textit{automizer} of $P$ in $G$, $\Aut_G(P) \cong N_G(P)/C_G(P)$, the image of the homomorphism $N_G(P) \to \Aut(P), g\mapsto c_g$, where $c_g \in \Aut(P)$ denotes the map induced by conjugation by $g$.

    \begin{notation}
        We set $\calR_{G,k} :=  \left(\prod_{P \in s_p(G)} \Hom(N_G(P)/PC_G(P),k^\times)\right)'$ and set $\calR_G := \calR_{G,\F_p}$. We refer to $\calR_{G,k}$ as the \textit{reduced coherent character tuple} for $G$ over $k$.
    \end{notation}

    \begin{remark}\label{rmk:decomp}
        The above decomposition arises since $\calR_{G,k}$ is since the kernel of the projection map \[\pi_1: \left(\prod_{P \in s_p(G)} \Hom(N_G(P)/P,k^\times)\right)' \to \Hom(G,k^\times).\]
    \end{remark}

    \subsection{Compatibility of the Euler characteristic and decomposition}

    Next, we describe how the Euler characteristic of an endotrivial complex and the decomposition of $O(T(kG))$ are compatible.

    \begin{notation}
        Given a superclass function $f \in CF(G)$, define the \textit{dimension function} of $f$, denoted $\dim(f)$, as follows: \[\dim(f)(H) := (-1)^{f(H)}.\] Then $\dim(f)$ is also a superclass function, and the assignment $f \mapsto \dim(f)$ induces a group homomorphism \[\dim: CF(G) \to CF(G)^\times.\]

        Since $\kappa$ is an isomorphism, from here we identify $O(T(kG))$ with its decomposition, and write $\Lambda$ for the composition $\kappa \circ \Lambda$ when it is clear. Additionally, it will be convenient to identify $B(G)^\times$ with its image in $CF(G)^\times$. We also have an identification $(B(S)^G)^\times$ with its image in $(CF(S)^G)^\times$, which identifies with $CF(G,p)^\times$.
    \end{notation}

    \begin{theorem}\label{thm:etrivtodecomp}
        The map \[\Lambda: \calE_k(G) \to  (B(S)^G)^\times \times \left(\prod_{P \in s_p(G)} \Hom(N_G(P)/P, k^\times)\right)'\] is described as follows: \[[C] \mapsto \big(\dim (h_C),\calH(C)\big).\]
    \end{theorem}
    \begin{proof}
        This is a reformulation of \cite[Proposition 4.5]{M24a} and \cite[Proposition 4.6]{M24a}.
    \end{proof}

    Since $\calE_k(G)$ has a full-rank subgroup $\calT\calE(G)$ which lives over $\F_p$, combining Theorem \ref{thm:etrivtodecomp} and Proposition \ref{prop:furtherdecomp} obtains the following.

    \begin{theorem}
        The image of \[ \Lambda: \calE_k(G) \to  (B(S)^G)^\times \times \left(\prod_{P \in s_p(G)} \Hom(N_G(P)/P, k^\times)\right)'\] is a subgroup of \[(B(S)^G)^\times\times \Hom(G, k^\times) \times \calR_G.\] This further restricts to a homomorphism \[\Lambda: \calT\calE(G) \to (B(S)^G)^\times \times \calR_G.\]
    \end{theorem}
    \begin{proof}
        It suffices to compute the image of $\calT\calE(G)$, since we have a decomposition $\calE_k(G) =\calT\calE(G) \times \Hom(G,k^\times)$, and the latter component maps via the identity onto the $\Hom(G,k^\times)$ component in the target of $\Lambda$. The projection of $\Lambda(\calT\calE(G))$ onto the component of character tuples is a subgroup of $\ker(\pi_1) = \calR_{G,k}$ from Remark \ref{rmk:decomp}. Since the decomposition \[\left(\prod_{P \in s_p(G)} \Hom(N_G(P)/P,k^\times)\right)' \cong \Hom(G,k^\times) \times \left(\prod_{P \in s_p(G)} \Hom(N_G(P)/PC_G(P),k^\times)\right)'\] also arises from $\pi_1$, it follows that the projection of $\im( \Lambda)$ onto the component of character tuples is a subgroup of $\ker(\pi_1) = \calR_{G,k}.$ Now, since all indecomposable representatives of elements of $\calT\calE(G)$ descend to $\F_p$, the result follows.
    \end{proof}

    \section{Surjectivity of the Lefschetz homomorphism: the case of $p = 2$} \label{sec:p2}

    Our goal for this section is to compute the image of $\calE_{\F_2}(G) \to O(T(\F_2 G))$. In this case, $\F_2^\times = \{1\}$, so all groups of character tuples are trivial (which certainly is not true for arbitrary field extensions $k/\F_2$), and it only remains to compute the image of
    \begin{align*}
        \calE_{\F_2}(G) &\to (B(S)^G)^\times\\
        [C] &\mapsto \dim(h_C).
    \end{align*}

    If $G$ is a $p$-group, this map is surjective, which follows from an unpublished topological result of Tornehave \cite{Tor84} regarding the \textit{tom Dieck homomorphism}, which later was proved algebraically by Yal\c{c}in \cite[Corollary 1.2]{Y05}.

    \begin{theorem}{\cite[Proposition 7.4]{M24c}}
        Let $G$ be a $p$-group. We have \[im(\dim: CF_b(G) \to CF(G)^\times ) = \im(m: B(G)^\times \to CF(G)^\times).\] Therefore, we may regard the dimension function as a surjective group homomorphism $\dim: CF_b(G) \to B(G)^\times$, and in particular, $\Lambda: \calE_k(G) \to O(T(kG))$ is surjective.
    \end{theorem}

        Therefore, given any unit $u \in (B(S)^G)^\times$, there exists some Borel-Smith function $f \in CF_b(S)$ such that $\dim(f) = u$. Note that the only interesting case is $p=2$, since if $p$ is odd, $|B(G)^\times| = 2$. However, $f$ is not guaranteed to be $G$-stable, i.e. there is no guarantee that $f \in CF_b(S)^G = CF_b(G,p) $, i.e. that $f$ extends to a Borel-Smith function on $G$. Effectively, we have the following rewording: \[\text{$\Lambda: \calE_{\F_2}(G) \to O(T(\F_2G))$ is surjective if and only if $\dim: CF_b(G,p) \to (B(S)^G)^\times$ is surjective}.\]

    \begin{remark}
        Tornehave's theorem \cite{Tor84} explicitly proves that the tom Dieck homomorphism $\eta: R_\R(G) \to B(G)^\times$ is surjective for all 2-groups $G$. However, $\eta$ is known to not be surjective for non-2-groups, for instance $G = A_5$ \cite[Remark 5.3]{Bar11}. Additionally, it does not seem easy to adopt Tornehave's original proof, as it involves a reduction step that may not be compatible with $G$-stability for a Sylow $p$-subgroup of $G$, and Yal\c{c}in's algebraic proof uses a restriction to so-called \textit{genetic subgroups} of $G$, which may not be compatible with $G$-stability either. At present, the question of surjectivity for $p = 2$ seems a difficult task, and a pertinent, interesting question.
    \end{remark}

    \subsection{$2$-fusion controlled by normalizers}

    We first recall a few necessary preliminaries of fusion systems and control of fusion, following \cite{AKO11}.

    \begin{definition}
        Let $G$ be a finite group. Recall that for any two subgroups $H, K$ of $G$, $\Hom_G(H,K)$ denotes the set of all (necessarily injective) homomorphisms from $H$ to $K$ induced by conjugation in $G$. We denote the set of all injective homomorphisms from $H$ to $K$ by $\operatorname{Inj}(P,Q)$.

        A \textit{fusion system} over a $p$-group $S$ is a category $\calF$ with objects the set of all subgroups of $S$ and which satisfy the following two properties for all subgroups $P, Q$ of $S$.
        \begin{itemize}
            \item $\Hom_S(P,Q) \subseteq \Hom_\calF(P,Q) \subseteq \operatorname{Inj}(P,Q)$;
            \item Each $\phi \in \Hom_\calF(P,Q)$ is the composition of an $\calF$-isomorphism followed by an inclusion.
        \end{itemize}
        Let $S$ be a Sylow $p$-subgroup of a finite group $G$. We denote by $\calF_S(G)$ the fusion system over $S$ with hom-sets $\Hom_{\calF_S(G)}(P,Q) := \Hom_G(P,Q)$ for all subgroups $P, Q$ of $S$. This fusion system is called \textit{the fusion category of $G$ over $S$}. A fusion system $\calF$ over $S$ is \textit{realizable} if $\calF = \calF_S(G)$ for some finite group $G$ with Sylow $p$-subgroup $S$.

        Let $H$ be a subgroup of $G$ containing a Sylow $p$-subgroup $S$. We say \textit{$H$ controls fusion in $S$} if $\calF_S(G) = \calF_S(H)$. Explicitly, this implies that if there exists $g\in G$ inducing an isomorphism $c_g: P \to Q$ for two subgroups $P, Q$, then there exists $h \in H$ such that $c_g = c_h$.
    \end{definition}

    Fusion system-theoretic considerations will be used minimally in this paper. For further reading, we refer the reader to \cite{AKO11} for a comprehensive overview.

    \begin{remark}
        If $S \trianglelefteq G$, then that the set $CF_b(S)$ has $G$-set structure induced by $G$-conjugation. Then $CF_b(S)^G$ is exactly the set of elements fixed by $G$. Given any Borel-Smith function $f \in CF_b(S)$, \[\tr^G_S f := \sum_{g \in [G/S]} {}^gf \in CF_b(S)^G\] is a $G$-stable Borel-Smith function on $S$, and therefore extends to a Borel-Smith function on $G$.

        More generally, if $N_G(S)$ controls fusion of $S$, then $\tr^{N_G(S)}_S f$ is a $N_G(S)$-stable Borel-Smith function, hence $G$-stable, and therefore also extends to a Borel-Smith function on $G$. Note we may assume a weaker condition - we only require that any two $G$-conjugate subgroups of $S$ are $N_G(S)$-conjugate.
    \end{remark}

    \begin{theorem}\label{thm:fusioncontrolledbynormalizer}
        Let $G$ be a finite group with Sylow $2$-subgroup $S$ for which $N_G(S)$ controls fusion in $S$, i.e. $\calF_S(G) \cong \calF_S(N_G(S))$. Then the group homomorphism $\Lambda: \calE_{\F_2}(G) \to O(T(\F_2 G))$ is surjective.

        In fact, this holds for any group $G'$ with a Sylow $2$-subgroup isomorphic to $S$ satisfying $\calF_S(G) \cong \calF_S(G')$.
    \end{theorem}
    \begin{proof}
        Let $u \in (B(S)^G)^\times$. There exists some Borel-Smith function $f \in CF_b(S)$ such that $\dim(f) = u$. In particular, for all $G$-conjugate subgroups $H,K$, we have $f(H) \equiv f(K) \mod 2$. Therefore, \[\left(\tr^{N_G(S)}_S f\right)(P) \equiv [N_G(S):S] f(P) \equiv f(P) \mod 2\] for all subgroups $P$ of $S$. Since $N_G(S)$ controls fusion in $S$, $\tr^{N_G(S)}_S f$ is an $N_G(S)$-stable, hence $G$-stable, Borel-Smith function, and therefore extends to a Borel-Smith function on $G$. Since $\tr^{N_G(S)}_S f \equiv f \mod 2$, we have $\dim\left(\tr^{N_G(S)}_S f\right) = u$. Now, there exists a unique endotrivial complex $C$ of $\F_2 G$-modules with h-marks $h_C = f$, and $\dim(h_C) = u \in (B(S)^G)^\times$, as desired. The last statement follows since $\tr^G_S f$ as before can be regarded as a $G'$-stable Borel-Smith function, hence a Borel-Smith function on $G'$ as well.
    \end{proof}

    \begin{remark}
        In particular, $\Lambda$ is injective for any finite group $G$ with a \textit{resistant} Sylow 2-subgroup. A \textit{resistant} $p$-group $P$ satisfies that for any finite group $G$ with Sylow $p$-subgroup $P$, $N_G(P)$ controls fusion in $P$. Burnside's theorem implies all abelian $p$-groups are resistant. Stancu proves in \cite{St02} that almost all generalized extraspecial $p$-groups are resistant, with the lone exceptions arising from $p$-groups of the form $E \times A$, where $E$ is a Sylow $p$-subgroup of $\GL_3(p)$ (which has self-normalizing Sylow $p$-subgroups), and $A$ is elementary abelian.
    \end{remark}

    \subsection{Dihedral Sylow $2$-subgroups}

    The situation in characteristic $2$, in the case when a normalizer of a Sylow $2$-subgroup does not control fusion, remains unclear, and it appears to be a difficult question in general.

    We demonstrate that the Euler characteristic $\Lambda$ is surjective for any group with a dihedral Sylow $2$-subgroup $D_{2^n}$ satisfying $n \geq 3$. Note that the cases $n = 1,2$ are covered by Theorem \ref{thm:fusioncontrolledbynormalizer}. This includes the finite simple groups $A_7$ and $\operatorname{PSL}_2(q)$ with $q \geq 5$ an odd prime power, see \cite{GW65}. Note this list includes $A_5 \cong \operatorname{PSL}_2(5)$ and $A_6 \cong \operatorname{PSL}_2(9)$.

    \begin{remark}\label{rmk:described}
        To verify this statement, it suffices to verify it for all realizable fusion systems $\calF$ on $D_{2^n}$. The classification of these is known; we review this now, following \cite{Li07}. We describe $S := D_{2^n}$ as follows:
        \[S := \langle x, t \mid x^{2^{n-1}} = t^2 = 1, {}^tx = x\inv\rangle.\] Set $z = x^{2^{n-2}}$, then $\langle z \rangle = Z(S)$. Besides the trivial fusion system $\calF_S(S)$, there are two other fusion systems, both of which are realized by finite groups. We denote by $\calF_S^1$ the fusion system on $S$ generated by $\calF_S(S)$ and an automorphism of order 3 of the Klein-four group $\langle z\rangle \times \langle t\rangle$. In this case $z$ and $t$ are $\calF^I_S$-conjugate while $z$ and $xt$ are not. We denote by $\calF_S^2$ the fusion system on $S$ generated by $\calF_S(S)$ and an automorphism of order 3 on each of the Klein-four groups $\langle z \rangle \times \langle t\rangle$ and $\langle z \rangle \times \langle xt\rangle$. Therefore, all elements of $S$ of order 2 are $\calF_S^{II}$ conjugate. Both $\calF_S^I$ and $\calF_S^{II}$ do not fuse any subgroups of order greater than 2 that are not already fused by $\calF_S(S)$.

        The only simple fusion system on $S$ is $\calF_S^{II}$. All three fusion systems arise from finite groups. For instance, if $q$ is an odd prime power satisfying $q \equiv \pm 1 \mod 8$, then $\operatorname{PSL}_2(q)$ has a dihedral Sylow 2-subgroup $S$ with $\calF_S(\operatorname{PSL}_2(q)) \cong \calF_S^{II}$. If $q$ is an odd prime power satisfying $q \equiv \pm 3 \mod 8$, then $\operatorname{PGL}_2(q)$ has a dihedral Sylow 2-subgroup $S$ with $\calF_S(\operatorname{PGL}_2(q)) \cong \calF_S^I$.

        \begin{figure}[H]
            \centering
            \begin{tikzcd}
            	&& {D_8} \\
            	& {V_4} & {C_4} & {V_4} \\
            	{C_2} & {C_2} & {Z(D_8)} & {C_2} & {C_2} \\
            	&& 1
            	\arrow[no head, from=1-3, to=2-2]
            	\arrow[no head, from=1-3, to=2-3]
            	\arrow[no head, from=1-3, to=2-4]
            	\arrow[no head, from=2-2, to=3-1]
            	\arrow[no head, from=2-2, to=3-2]
            	\arrow[no head, from=2-2, to=3-3]
            	\arrow[no head, from=2-3, to=3-3]
            	\arrow[no head, from=2-4, to=3-3]
            	\arrow[no head, from=2-4, to=3-4]
            	\arrow[no head, from=2-4, to=3-5]
            	\arrow[dashed, no head, from=3-1, to=3-2]
            	\arrow[dotted, no head, from=3-2, to=3-3, "I"]
            	\arrow[no head, from=3-2, to=4-3]
            	\arrow[dotted, no head, from=3-3, to=3-4, "II"]
            	\arrow[dashed, no head, from=3-4, to=3-5]
            	\arrow[no head, from=4-3, to=3-1]
            	\arrow[no head, from=4-3, to=3-4]
            	\arrow[no head, from=4-3, to=3-5]
                \arrow[no head, from=4-3, to=3-3]
            \end{tikzcd}
            \caption{The three fusion systems associated to $D_8$. The fusion system $\calF_{D_8}^I$ fuses the subgroups connected by the dotted $I$, and the fusion system $\calF_{D_8}^{II}$ fuses the subgroups connected by both the dotted $I$ and $II$. The picture is analogous for all $D_{2^n}$ with $n \geq 3$, as no subgroups of greater order are fused.}
        \end{figure}
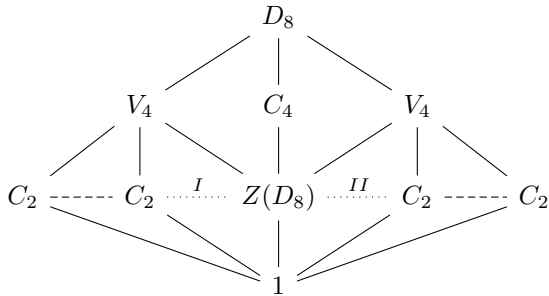

        We next recall $\calE_k(S) \cong CF_b(S)$ and $B(S)^\times$. Since $CF_b(-)$ is a biset functor, we have a decomposition \[CF_b(S) = \bigoplus_{N \trianglelefteq S} \partial \Inf^S_{S/N} CF_b(S/N),\] see \cite[Lemma 6.3.2]{Bou10}. The normal subgroups of $S = D_{2^n}$ are the two noncyclic subgroups of index two, isomorphic to $D_{2^{n-1}}$, which we denote by $H_1$ and $H_2$, and the subgroups $C_{2^i}$ of the maximal cyclic subgroup of $S$ (of index $2$), for $0\leq i\leq 2^{n-1}$. We have that $D_{2^n}/C_{2^i} \cong D_{2^{n-i}}$ for $0 \leq i \leq 2^{n-3}$. For $n \geq 3$, we have that $\partial\calE_k(D_{2^n}) = \langle C^{n}\rangle$, where $ C^n$ is an endotrivial complex satisfying $h_{C^n}(1) = 2$, $h_{C^n}(H) = 1$ where $H$ is a noncentral subgroup of $D_{2^n}$ of order 2, and $h_{C^n}(K) = 0$ for all other subgroups $K$, see \cite[Theorem 6.9]{M24a}. Note that $\partial \calE_k(V_4) = \emptyset$.

        It follows that $CF_b(S)$ has a $\Z$-basis as follows, where we denote by $f_N$ the inflated faithful Borel-Smith function on $S/N$. We denote by $K_i$ for $i \in \{1, 2\}$ the unique up to conjugacy noncentral cyclic subgroup of order 2 satisfying $K_i \leq H_i$. Note $f_{C^{2^{n-2}}}$ does not exist. The tables for $D_8$ and $D_{2^n}$ for $n \geq 4$ are displayed separately, as they differ slightly. For the latter table, we only display the h-marks at the subgroups of order at most 2, as these are the only subgroups whose fusion creates complications.

        \begin{table}[H]
            \centering
            \begin{tabular}{|c|c|c|c|c|c|c|c|c|}
            \hline
               & $1$ & $K_1$ & $K_2$ & $C_2$ & $H_1$ & $H_2$ & $C_4$ & $D_8$ \\
                \hline
              $f_{D_8}$ & 1 & 1 & 1 & 1 & 1 & 1 & 1 & 1 \\
              $f_{H_1}$ & 1 & 1 & 0 & 1 & 1 & 0 & 0 & 0\\
              $f_{H_2}$ & 1 & 0 & 1 & 1 & 0 & 1 & 0 & 0 \\
              $f_{C_4}$ & 1 & 0 & 0 & 1 & 0 & 0 & 1 & 0 \\
              $f_1$ & 2 & 1 & 1 & 0 & 0 & 0 & 0 & 0\\
                \hline
            \end{tabular}
            \caption{The $\Z$-basis of $CF_b(D_8)$ arising from the real representation ring $R_\R(D_8)$, which coincides with the $\Z$-basis arising from the canonical decomposition into faithful constituents.}
        \end{table}

        \begin{table}[H]
            \centering
            \begin{tabular}{|c|c|c|c|c|c|c|c|c|}
            \hline
               & $1$ & $K_1$ & $K_2$ & $C_2$ \\
                \hline
              $f_{D_{2^n}}$ & 1 & 1 & 1 & 1  \\
              $f_{H_1}$ & 1 & 1 & 0 & 1  \\
              $f_{H_2}$ & 1 & 0 & 1 & 1  \\
              $f_{C_{2^{n-1}}}$ & 1 & 0 & 0 & 1 \\
              $f_{C_{2^{n-3}}}$ & 2 & 1 & 1 & 2 \\
              $\vdots$&$\vdots$&$\vdots$&$\vdots$&$\vdots$ \\
              $f_{C_2}$ & 2 & 1 & 1 & 2  \\
              $f_1$ & 2 & 1 & 1 & 0   \\
                \hline
            \end{tabular}
            \caption{The $\Z$-basis of $CF_b(D_{2^n})$ for $n \geq 4$ arising from the canonical decomposition into faithful constituents. Only values for subgroups of order at most 2 are displayed. Note $u_{C_{2^{n-3}}}, \dots, u_{C_2}, u_1$ are distinct, their values differ at subgroups of higher order.}
        \end{table}

        Tornehave/Yal\c{c}in's theorem implies that the exponential map $\dim: CF_b(S) \to B(S)^\times$ is surjective for any $p$-group $S$. This allows us to easily deduce the unit group of the Burnside ring (via its marks) for all dihedral $2$-groups $D_{2^n}$ for $n \geq 3$. We again display these separately for $D_8$ and $D_{2^n}$ for $n \geq 4$. We remark that $\ker(\dim)$ is a finite group for all dihedral groups $D_{2^n}$, despite the appearance of the latter table.

        \begin{table}[H]
            \centering
            \begin{tabular}{|c|c|c|c|c|c|c|c|c|}
            \hline
               & $1$ & $K_1$ & $K_2$ & $C_2$ & $K_1$ & $K_2$ & $C_4$ & $D_8$ \\
                \hline
              $u_{D_8}$ & -1 & -1 & -1 & -1 & -1 & -1 & -1 & -1 \\
              $u_{H_1}$ & -1 & -1 & 1 & -1 & -1 & 1 & 1 & 1\\
              $u_{H_2}$ & -1 & 1 & -1 & -1 & 1 & -1 & 1 & 1 \\
              $u_{C_4}$ & -1 & 1 & 1 & -1 & 1 & 1 & -1 & 1 \\
              $u_1$ & 1 & -1 & -1 & 1 & 1 & 1 & 1 & 1\\
                \hline
            \end{tabular}
            \caption{The $\F_2$-basis of $B(D_8)^\times$, viewed via its marks.}
        \end{table}

        \begin{table}[H]
            \centering
            \begin{tabular}{|c|c|c|c|c|c|c|c|c|}
            \hline
               & $1$ & $K_1$ & $K_2$ & $C_2$ \\
                \hline
              $u_{D_{2^n}}$ & -1 & -1 & -1 & -1  \\
              $u_{H_1}$ & -1 & -1 & 1 & -1  \\
              $u_{H_2}$ & -1 & 1 & -1 & -1  \\
              $u_{C_{2^{n-1}}}$ & -1 & 1 & 1 & -1 \\
              $u_{C_{2^{n-3}}}$ & 1 & -1 & -1 & 1 \\
              $\vdots$&$\vdots$&$\vdots$&$\vdots$&$\vdots$ \\
              $u_{C_2}$ & 1 & -1 & -1 & 1  \\
              $u_1$ & 1 & -1 & -1 & 1   \\
                \hline
            \end{tabular}
            \caption{The $\F_2$-basis of $B(D_{2^n})^\times$ for $n \geq 4$, viewed via its marks at subgroups of order at most 2. Note $u_{C_{2^{n-3}}}, \dots, u_{C_2}, u_1$ are distinct, their marks differ at subgroups of higher order.}
        \end{table}

    \end{remark}

    \begin{theorem}\label{thm:dihedralsurj}
        Let $G$ be a finite group with a dihedral Sylow $2$-subgroup $S$ of order $2^n$ for some $n\geq3$. Then the restricted dimension homomorphism \[\dim: CF_b(G,p)\to (B(S)^G)^\times\] is surjective.
    \end{theorem}
    \begin{proof}
        We directly compute this in cases. Set $\calF := \calF_S(G)$. If $\calF = \calF_S(S)$ is the trivial fusion system then this already holds by Tornehave/Yal\c{c}in's theorem.

        Next, suppose $\calF = \calF_S^I$. We assume without loss of generality that $K_1$ and $Z(S) = C_2$ are fused, using the conventions of Remark \ref{rmk:described}. If $S \cong D_8$, then direct computation shows $CF_b(G)^\calF$ has a $\Z$-basis given by \[\{f_{D_8}, f_{H_1}, f_{H_2} - f_{C_4}, f_1 + f_{C_4}\}\] and $(B(G)^\times)^\calF$ has an $\F_2$-basis given by \[\{u_{D_8}, u_{H_1}, u_{H_2} + u_{C_4}, u_1 + u_{C_4}\},\] so the dimension homomorphism is evidently surjective. If $S \cong D_{2^n}$ for $n \geq 4$, then direct computation shows $CF_b(G)^\calF$ has a $\Z$-basis given by \[\{f_{D_{2^n}}, f_{H_1}, f_{H_2} - f_{C_{2^{n-1}}}, f_{C_{2^{n-3}}} - f_{C_{2^{n-1}}}, \dots , f_{C_2} - f_{C_{2^{n-1}}}, f_1 + f_{C_{2^{n-1}}} \}\] and $(B(G)^\times)^\calF$ has an $\F_2$-basis given by \[\{u_{D_{2^n}}, u_{H_1}, u_{H_2} + u_{C_{2^{n-1}}}, u_{C_{2^{n-3}}} + u_{C_{2^{n-1}}}, \dots , u_{C_2} + u_{C_{2^{n-1}}}, u_1 + u_{C_{2^{n-1}}} \}\] so again the dimension homomorphism is surjective.

        Now, suppose $\calF = \calF_S^{II}$, so both $K_1$ and $K_2$ are fused with $Z(S) = C_2$. If $S \cong D_8$, then direct computation shows $CF_b(G)^\calF$ has a $\Z$-basis given by \[\{f_{D_8}, f_{H_1} + f_{H_2} - f_{C_4}, f_1 + f_{C_4}\}\] and $(B(G)^\times)^\calF$ has an $\F_2$-basis given by \[\{u_{D_8}, u_{H_1} + u_{H_2} + u_{C_4}, u_1 + u_{C_4}\},\] so the dimension homomorphism is evidently surjective. If $S \cong D_{2^n}$ for $n \geq 4$, then direct computation shows $CF_b(G)^\calF$ has a $\Z$-basis given by \[\{f_{D_{2^n}}, f_{H_1} + f_{H_2} - f_{C_{2^{n- 1}}}, f_{C_{2^{n-3}}} - f_{C_{2^{n- 1}}}, \dots, f_{C_2} - f_{C_{2^{n- 1}}}, f_1 + f_{C_{2^{n- 1}}} \}\], and $(B(G)^\times)^\calF$ has an $\F_2$-basis given by \[\{u_{D_{2^n}}, u_{H_1} + u_{H_2} + u_{C_{2^{n- 1}}}, u_{C_{2^{n-3}}} - u_{C_{2^{n- 1}}}, \dots, u_{C_2} - u_{C_{2^{n- 1}}}, u_1 + u_{C_{2^{n- 1}}}\},\] so the dimension homomorphism again is surjective.
    \end{proof}

    \begin{remark}
        We observe that for $S$ dihedral of order at least $8$, $CF_b(G)^{\calF^{I}_S}$ and $(B(G)^\times)^{\calF_S^{I}}$ are maximal subspaces of $CF_b(G)$ and $B(G)^\times$ respectively, and both $CF_b(G)^{\calF^{II}_S}$ and $(B(G)^\times)^{\calF_S^{II}}$ are maximal subspaces of $CF_b(G)^{\calF^{I}_S}$ and $(B(G)^\times)^{\calF_S^{I}}$ respectively.
    \end{remark}

    From Theorem \ref{thm:dihedralsurj} (and Theorem \ref{thm:fusioncontrolledbynormalizer} for dihedral groups of order less than $8$), we immediately obtain the following corollary.
    \begin{corollary}
        Let $G$ be a finite group with a dihedral Sylow $2$-subgroup. Then the Euler characteristic map $\Lambda: \calE_{\F_2}(G) \to O(T(\F_2G))$ is surjective.
    \end{corollary}

    \section{Surjectivity of the Lefschetz homomorphism: the case of $p$ odd} \label{sec:podd}

    Contrary to the $p = 2$ case, the $p$ odd case revolves around the coherent character tuple \[\left(\prod_{P \in s_p(G)} \Hom(N_G(P)/P,k^\times)\right)' \leq O(T(kG)),\] since for odd primes $p$, $B(P)^\times = \{\pm[P/P]\}$ for any $p$-group $P$. Recall we have the decomposition \[ \left(\Hom(N_G(P)/P,k^\times)\right)'  = \Hom(G,k^\times) \times \left(\prod_{P \in s_p(G)} \Hom(N_G(P)/PC_G(P),k^\times)\right)'.\]  From this decomposition, we record a straightforward result.

    \begin{theorem}
        Let $p$ be an odd prime, $k$ any field of characteristic $p$, and $G$ a $p$-nilpotent group. Then $\Lambda: \calE_k(G) \to O(T(kG))$ is surjective.
    \end{theorem}
    \begin{proof}
        This follows from \cite[Proposition 6.3]{BC23}, we spell out the details here. The Frobenius normal $p$-complement theorem (see e.g. \cite[Theorem 7.4.5]{Gor68}) states that a group $G$ is $p$-nilpotent if and only if $N_G(P)/C_G(P)$ is a $P$-group for all $p$-subgroups $P$ of $G$. Therefore, $\ker(\pi_1)$ is trivial, and under the decomposition $O(T(kG)) = \Z/2\Z \times \Hom(G, k^\times)$, it is routine to check that $\Lambda$ is surjective.
    \end{proof}

    \begin{remark}
        Similarly, for $p$ odd, any group $G$ with $|G|$ coprime to $p-1$ satisfies $O(T(\F_pG)) \cong (B(S)^\times)^G = \Z/2\Z$; therefore $\Lambda$ will be surjective over $\F_p$, but not necessarily over extensions $k/\F_p$. This is an edge case, as $p$-local structure is not relevant.
    \end{remark}

    Recall that for any field extension $k/\F_p$, we have an identification $\calT\calE_k(G) \cong \calT\calE_{\F_p}(G)$. For the general case, $\Lambda: \calE_k(G) \to O(T(kG))$ restricts to a homomorphism \[\Lambda: \calT\calE_k(G) = \calT\calE_{\F_p}(G) \to \Z/2\Z \times \left(\prod_{P \in s_p(G)} \Hom(N_G(P)/PC_G(P),\F_p^\times)\right)',\] since $\calT\calE(G) = \{ C\in \calE_k(G) \mid\calH_C(1) = 1 \in \Hom(G,k^\times) \}$. It suffices to determine the image of $\Lambda$ in $\ker(\pi_1)$, since the endotrivial complex $k[1]$ satisfies $\Lambda(k[1]) = -[G/G] \in (B(S)^G)^\times$ and $\calH(k[1]) = (1)_{P\in s_p(G)}$. Moreover, for any generator $C$ not $k[1]$ of $\calT\calE(G)$ (with respect to the basis of $\calT\calE(G) \cong CF_b(G)$ associated to the real irreps of $G$), $C$ has even h-marks, so the projection of $\Lambda(C)$ onto $(B(S)^G)^\times$ is trivial.

    From Corollary \ref{cor:etrivclassification}, $\calT\calE(G)$ has at minimum $c(G)$ generators, with one of those generators, $k[1]$, belonging to $\ker(\pi_1)$. Set $\calT\calE'(G) := \calT\calE(G)/\langle k[1]\rangle$, then we have a group homomorphism $\Lambda'$ induced from $\Lambda$ as follows \[\Lambda': \calT\calE'(G) \to \left(\prod_{P \in s_p(G)} \Hom(N_G(P)/PC_G(P),\F_p^\times)\right)'.\]
    Clearly, $\Lambda'$ is surjective if and only if $\Lambda: \calE_{\F_p}(G) \to O(T(\F_pG))$ is surjective. Since $\calT\calE'(G)$ has $\Z$-rank $c(G) - 1$, the following statement follows easily.

    \begin{observation}\label{obs:gens}
        Let $p$ be an odd prime and $k$ any field of characteristic $p$. If the group of reduced coherent character tuples $\calR_G$ has a minimum generating set (with respect to size) of at least $c(G)$ elements, then $\Lambda: \calE_k(G) \to O(T(kG))$ is not surjective.
    \end{observation}

    However, explicitly determining the structure of $\calR_G$ is difficult in general.

    \subsection{Cyclic Sylow $p$-subgroups}

    In this section, we assume $k$ is a field extension of $\F_p$ contained in $\overline{\F}_p$. In particular, $\Aut(k)$ has a unique cyclic subgroup of order $p-1$. We turn to perhaps the most natural case to consider, groups with cyclic Sylow $p$-subgroups, with our focus on the map $\Lambda'$. First, we make an easy but critical observation: the coherence condition plays no significant role for $\Lambda'$. This is an extension of \cite[Corollary 6.2]{BC23}, where the phenomenon is observed for groups with normal cyclic Sylow $p$-subgroups.

    \begin{prop}\label{prop:nocoherenceforcyclics}
        Let $G$ be a group with cyclic Sylow $p$-subgroup $S$. Let $(\chi_P)_{P \in s_p(G)} \in \left(\prod_{P \in s_p(G)} \Hom(N_G(P)/PC_G(P),k^\times)\right)'$. Let $Q \leq S$ be a subgroup of $S$ and let $x \in N_G(Q)$ have unique decomposition $x = x_p x_{p'} = x_{p'}x_{p}$ into $p$- and $p'$-elements. We have $\chi_{\langle x_p \rangle}(x) = 1$ and $x \in C_G(\langle x_p \rangle).$ In particular, \[\left(\prod_{P \in s_p(G)} \Hom(N_G(P)/PC_G(P),k^\times)\right)' = \left(\prod_{P \in s_p(G)} \Hom(N_G(P)/PC_G(P),k^\times)\right)^G.\]
    \end{prop}
    \begin{proof}
        By assumption, $x_p' \in C_G(x_p)$ and hence $x_p' \in C_G(\langle x_p \rangle)$. Obviously $x_p \in C_G(\langle x_p \rangle)$ as well, therefore $x \in C_G(\langle x_p \rangle)$ and $\chi_{\langle x_p \rangle}(x) = 1$ since $\chi_{\langle x_p \rangle}\in \Hom(N_G(\langle x_p \rangle)/C_G(\langle x_p \rangle), \F_p^\times)$. Now, if $Q < Q\langle x_p\rangle = \langle x_p \rangle$, then the coherence condition implies $\chi_Q(x) = \chi_{\langle x_p \rangle}(x) = 1$, which is redundant information since $x \in C_G(\langle x_p \rangle)$, hence $x \in C_G(Q)$ as well.
    \end{proof}

    \begin{remark}\label{rmk:cohomologystuff}
        Recall (see \cite{Sw59}) that for $p$ odd, the groups with periodic cohomology over fields of characteristic $p$ are precisely those with cyclic Sylow $p$-subgroups. If $G$ has cyclic Sylow $p$-subgroup of order $p^r$, then $G$ has period $2\Phi(S)$, where $\Phi(S)$ denotes the number of automorphisms of $S$ that are given by conjugation by elements of $G$, i.e. the order of $\Aut_G(S)\cong N_G(S)/C_G(S)$. Moreover, we have an inclusion \[N_G(S)/C_G(S) \cong \Aut_G(S) \leq \Aut(S) \cong (\Z/p^r\Z)^\times\cong C_{(p-1)p^{r-1}}.\]  In particular, $\Aut_G(S)$ is cyclic. Since $S$ is a Sylow $p$-subgroup, $\Aut_G(S)$ is a $p'$-group, therefore $\Aut_G(S)\hookrightarrow O_{p'}(\Aut(S)) \cong C_{p-1}$. It follows that \[\Hom(N_G(S)/C_G(S), k^\times) \cong N_G(S)/C_G(S),\] as the homomorphisms all factor as endomorphisms of $N_G(S)/C_G(S)$ followed by an inclusion $N_G(S)/C_G(S) \hookrightarrow k^\times$, with $O_{p'}(k^\times) \cong C_{p-1}$. Since $N_G(S)/C_G(S)$ is cyclic, its group of endomorphisms identifies with itself.

        In fact, for any nontrivial $p$-subgroup $P$ of $G$ contained in $S$, we have a similar inclusion $N_G(P)/C_G(P) \cong \Aut_G(P)\leq  \Aut(P)$ with $\Aut_G(P)$ contained in the largest $p'$-subgroup of $\Aut(P)$, which again is isomorphic to $C_{p-1}$.

        On the other hand, every automorphism of $S$ induces an automorphism on $P$, and elementary number theory implies every automorphism of $P$ arises in this way. Therefore, we have a surjective group homomorphism $\Aut(S) \to \Aut(P)$ with kernel a $p$-subgroup of $\Aut(S)$. In particular, the composition \[N_G(S)/C_G(S) \hookrightarrow \Aut(S)\twoheadrightarrow \Aut(P)\] is injective as well. We have a commutative diagram as follows.
        \begin{figure}[H]
            \centering
            \begin{tikzcd}
                N_G(S) \ar[hookrightarrow, r] \ar[d, twoheadrightarrow] & N_G(P) \ar[r] \ar[d, twoheadrightarrow] & \Aut(P) \\
                \Aut_G(S) \ar[r, dotted] & \Aut_G(P) \ar[ur, hookrightarrow]
            \end{tikzcd}
        \end{figure}
        Because the composition $\Aut_G(S) \to \Aut(P)$ is injective, it follows that the dotted arrow is also injective. In particular, we have an inequality \[[N_G(S):C_G(S)] \leq [N_G(P):C_G(P)].\]

        On the other hand, Burnside's fusion theorem (see \cite[Proposition 4.5]{AKO11}) implies that $\calF_S(G) = \calF_S(N_S(G))$, therefore any automorphism $\varphi \in \Aut_G(P)$ is of the form $c_g$ for some $g \in N_G(S)$. Thus the map $\Aut_S(G) \to \Aut_S(P)$ is surjective as well, so for all nontrivial $p$-subgroups $P$ of $S$, we have an isomorphism $\Aut_G(S) \cong \Aut_G(P)$. We summarize this now.
    \end{remark}

    \begin{prop}\label{prop:automizers}
        Let $p$ be a prime and let $G$ be a $p$-group with cyclic Sylow $p$-subgroup $S$. For all nontrivial $p$-subgroups $P$ of $S$, we have an isomorphism \[\Aut_G(S) \cong \Aut_G(P)\] induced by restriction of $c_g \in \Aut_G(S)$. Moreover, $\Aut_G(P)$ is a cyclic $p'$-group and if $p$ is odd, $G$ has period $2\Phi(S)$.
    \end{prop}
    \begin{proof}
        If $p$ is odd, this is proven in Remark \ref{rmk:cohomologystuff}. Otherwise if $p = 2$, this trivially holds since in this case, $N_G(S)/C_G(S)$ is a $2'$-group but $\Aut(\Z/2^r\Z)$ is a $2$-group, hence $\Aut_G(P)$ is trivial for any $p$-subgroup $P$ of $G$.
    \end{proof}

    We now will prove $\Lambda'$, hence $\Lambda$, is surjective. The proof will rely on induction on the order of Sylow $p$-subgroups: we provide the base case as an individual lemma.

    \begin{lemma}\label{lem:orderp}
        Let $G$ be a finite group with Sylow $p$-subgroups having order $p$. Then $\Lambda: \calE_{k}(G) \to O(T(kG))$ is surjective.
    \end{lemma}
    \begin{proof}
        It suffices to show $\calT\calE(G) \to \calR_G$ is surjective. Let $S$ be a Sylow $p$-subgroup of $G$. By Proposition \ref{prop:nocoherenceforcyclics}, $\calR_G \cong \Hom(N_G(S)/C_G(S), \F_p^\times)$, since the 1-term in $\calR_G$ (which is trivial) has no effect. Moreover, from Remark \ref{rmk:cohomologystuff}, $\Hom(N_G(S)/C_G(S), k^\times) \cong N_G(S)/C_G(S)$. Hence $|\calR_G| = [N_G(S):C_G(S)]$ as well.

        From the classification of endotrivial complexes, $\calT\calE(G)$ is generated by $k[1]$ and a three-term complex $C$ with h-mark 2 at the trivial subgroup 1 and h-mark 0 at all other $p$-subgroups. Therefore, $C$ generates $\calT\calE(G)'.$ Explicitly, we have $\Res^G_S C \simeq kS \to kS \to k$, a truncated periodic resolution of $k$ as $kS$-module - however $\calH_C(S) \neq k$ unless $[N_G(S):C_G(S)] = 1$.

        Since $G$ is periodic with period $2[N_G(S):C_G(S)]$, there is a chain complex $X$ that is a truncated periodic resolution of the trivial $k G$-module $\F_p$ with $2[N_G(S):C_G(S)] + 1$ nonzero terms. It follows that $X$ is endotrivial with h-mark $2[N_G(S):C_G(S)]$ at 1 and $0$ at $S$. Additionally, $X \in \ker(\Lambda') \leq \calT\calE(G)'$, and in fact $X$ must generate $\ker(\Lambda')$, since another shorter complex generating $\ker(\Lambda')$ would contradict the period of $G$. Finally $[N_G(S):C_G(S)]\cdot C = X \in \calT\calE(G)'$, as the h-marks of these complexes agree. Therefore we have an injective group homomorphism $\Lambda': \calT\calE(G)'/\langle X \rangle \hookrightarrow \calR_G$, and since both sides are finite groups of the same order, we conclude $\Lambda'$ is surjective, as desired.
    \end{proof}

    \begin{theorem}\label{thm:cyclicsurj}
        Let $G$ be a finite group with a cyclic Sylow $p$-subgroup $S$, and let $k$ be a field of characteristic $p$ contained in $\overline{\F}_p$. Then $\Lambda: \calE_{k}(G) \to O(T(kG))$ is surjective.
    \end{theorem}

    To prove the more general case, we first make a few observations. Let $G$ be a group with cyclic Sylow $p$-subgroups of order at least $p^2$. Let $S$ be a Sylow $p$-subgroup of $G$ and let $C \leq S$ be the unique subgroup of order $p$. Then for any subgroup $Q$ satisfying $C \leq Q \leq S$, we have $N_G(C) \geq N_G(Q) \geq N_G(S)$, and therefore $N_G(Q)/C_G(Q) = N_{N_G(C)}/C_{N_G(C)}(Q)$, and restriction induces an equality $\calR_G = \calR_{N_G(C)}$. Similarly, restriction induces an isomorphism $\calT\calE(G) \cong \calT\calE(N_G(C))$ (this holds for any group containing $S$). We have a commutative diagram as follows.

    \begin{figure}[H]
        \centering
        \begin{tikzcd}
            \calT\calE(G)' \ar[d, "\Lambda_G'"] \ar[r, "\cong"] & \calT\calE(N_G(C))' \ar[d, "\Lambda_{N_G(C)}'"]\\
            \calR_G  \ar[r, "\cong"]& \calR_{N_G(C)}
        \end{tikzcd}
    \end{figure}

    Therefore, $\Lambda_G$ is surjective if and only if $\Lambda_{N_G(C)}$ is surjective, and we can reduce to the case where $G$ has a normal cyclic subgroup $C$ of order $p$. Indeed, $G$ is balanced if $N_G(C)$ is balanced. Moreover, if $N_G(C)$ is balanced, then $N_G(C)/C$ is balanced.

    \begin{proof}[Proof (of Theorem \ref{thm:cyclicsurj})]
        We first show that if $G$ is balanced, then $\Lambda$ is surjective. It suffices to show that $\Lambda': \calT\calE(G) \to \calR_G = \calR_{G, k}$ is surjective, and we may assume $G$ has a normal cyclic subgroup $C$ of order $p$. We induct on the order of the Sylow $p$-subgroup $S \geq C$. The case $|S| = 1$ is trivial and the case $|S| = p$ is Lemma \ref{lem:orderp}.

        Under the assumption that $C$, observe that inflation induces a split injective group homomorphism $\calT\calE(G/C) \to \calT\calE(G)$ with retraction induced by the Brauer construction. This affords a decomposition \[\calT\calE(G) = \Inf^G_{G/C}\calT\calE(G/C) \times \partial\calT\calE(G),\] where $\partial\calT\calE(G)$ is generated by the unique ``faithful'' endotrivial complex $C$ with $h_C(1) = 2$ and $h_C(P) = 0$ for all nontrivial $p$-subgroups $P$ of $G$. We remark that this is a misnomer with regards to the language of \cite{M24a}, as for $C$ to be faithful, we would require $C(P) \cong k[0]$ for all $P$, where in this case we have $C(P) \cong k_\omega[0]$ for some $\omega \in \Hom(N_G(P)/P, \F_p^\times)$. However, it is easy to see that there is a linear character $\omega\in \Hom(G, \F_p^\times)$ such that $C\otimes_k k_\omega[0]$ is faithful in the sense of \cite{M24a}.

        Now, we characterize the images of $\Lambda'$ restricted to $\calT\calE(G/C)$ and $\partial\calT\calE(G)$. By induction, $\Lambda': \calT\calE(G/C)' \to \calR_{G/C}$ is surjective, since the assumption on $G$ holds for $G/C$ as well. We have an inflation map $\Inf^G_{G/C}: \calR_{G/C} \to \calR_G$. We claim its image is precisely the subgroup $\calR^C_G$ of $\calR_G$ given by tuples $(\chi_P)_P \in s_p(G)$ for which $\chi_C$ is trivial. It is easy to see the image of inflation is contained in $\calR_G^C$. On the other hand, for any $p$-subgroup of $G$ containing $C$, inflation sends a homomorphism $\chi_{P/C} \in \Hom(N_{G/C}(P/C)/C_{G/C}(P/C), k^\times)$ to $\chi_P \in \Hom(N_{G}(P)/C_G(P), k^\times)$. This map is an isomorphism since $\Aut_G(P)\cong \Aut_{G/C}(P/C)$, therefore $\Hom(N_{G/C}(P/C), k^\times) \cong \Hom(N_{G}(P), k^\times)$ and inflation is injective, hence surjective too. By induction, we have a surjective homomorphism \[\calT\calE(G/C) \xrightarrow{\calR_{G/C}}  \calR_{G}^C.\]

        Now observe that the $\calR^C_G$ has a complement $\Delta_C\calR_G \leq\calR_G$, which consists of all coherent character tuples $(\chi_P)_{P \in s_p(G)}$ for which $\chi_1 = 1$, $\chi_C = \omega$ for some $\omega \in \Hom(N_G(C)/C_G(C), k^\times)$, and for all $P\geq C$, $\chi_P = \omega|_{N_G(P)}$. It follows that $\calR^C_G \cong \Hom(N_G(C)/C_G(C),k^\times)$, hence $\Delta_C\calR_G$ has order $[N_G(C):C_G(C)] = [N_G(S):C_G(S)]$ by Proposition \ref{prop:automizers}. On the other hand, we have $\partial\calT\calE(G) = \langle C\rangle$ for an endotrivial complex $C$ with projectives in degrees 1 and 2, and the $k$-dimension one module $k_\omega$ for some $\omega \in \Hom(G, k^\times)$ in degree 0. Observe the image $\Lambda(C)$ is a tuple $(\chi_P)_{P \in s_p(G)}$ for which $\chi_1 = 1$, $\chi_C = \omega$, and for all $P\geq C$, $\chi_P = \omega|_{N_G(P)}$. Since the periodicity of $G$ is $n := 2\Phi(S) = 2[N_G(S):C_G(S)] $, as $C^{\otimes \Phi(S)}$ is a minimal length truncated periodic resolution of $k_\omega$ with homology again $k_\omega$, the image $\Lambda(C)$ in fact generates $\Delta_C\calR_G$, as desired.

    \end{proof}

    \subsection{Groups of $p$-rank at least two}

    On the other hand, there exist groups $G$ with small Sylow $p$-subgroups with $p$ odd for which, $\Lambda$ is not surjective. In fact, such groups need not have particularly complex Sylow $p$-subgroups, as the next example demonstrates.

    \begin{prop}\label{prop:prk2}
        Let $p$ be an odd prime and $k/\F_p$ be a field extension. Then there exists a finite group $G$ with normal Sylow $p$-subgroup $S$ elementary abelian of rank two such that $\Lambda: \calE_k(G) \to O(T(kG))$ is not surjective.
    \end{prop}
    \begin{proof}
        For $p>2$, the symmetric group $S_{2p}$ has elementary abelian Sylow $p$-subgroups, since the $p$-part of $(2p)!$ is $p^2$. We choose our Sylow to be $S := \langle (1, 2, \dots, p) ,(p+1, p+2, \dots, 2p)\rangle$. Set $a =  (1,2,\dots, p)$, $b = (p+1, p+2,  \dots, 2p)$, $A := \langle a\rangle $ and $B := \langle b\rangle $. We fix $G:= N_{S_{2p}}(S)$. We have that $G \cong (S \rtimes (C_{p-1}\times C_{p-1}))\rtimes C_2$, where for instance we may choose $C_2 = \langle(1, p+1) (2, p+2) \dots(p, 2p)\rangle$, although this choice is not unique. Generators of the copies of $C_{p-1}$ corresponding to $\Aut(A)$ and $\Aut(B)$ depend on $p$ on the other hand. Then $A$ and $B$ are $G$-conjugate and all other cyclic subgroups of $S$ are $G$-conjugate. We have $N_G(A) = N_G(B) \cong S \rtimes (\Aut(A) \times \Aut(B)) $, and if $C$ denotes any of the other $p-1$ cyclic subgroups of $S$ (i.e. $\langle ab\rangle, \langle ab^2 \rangle,\dots$) then $N_G(C) \cong (S \rtimes \Aut(C)) \rtimes C_2$, where $C_2$ is generated by an involution $\sigma$ satisfying ${}^\sigma a = b^n$, where $C = \langle ab^n\rangle$. Note $N_G(A) \trianglelefteq G$ with $G/N_G(A) \cong C_2$.

        We claim that $\calR_G$ has at least $c(G) = 3$ generators, which will imply by Observation \ref{obs:gens} that $\Lambda: \calE_k(G) \to O(T(kG))$ is not surjective. First, it is clear that $\calR_G$ has at least two generators, corresponding to generators for $\Hom(N_G(A)/C_G(A), k^\times)$ and $\Hom(N_G(C)/C_G(C),k^\times)$. Indeed, straightforward computations verify that these groups are nontrivial, with  both $N_G(A)/C_G(A) \cong C_{p-1}$ and $N_G(C)/C_G(C) \cong C_{p-1}$, therefore $\Hom(N_G(A)/C_G(A), k^\times) \cong C_{p-1}$ and $\Hom(N_G(C)/C_G(C),k^\times) \cong C_{p-1}$. Note there are no nontrivial coherence conditions implied from the trivial subgroup hom-set $\Hom(N_G(1)/C_G(1), k^\times)$.

        Note the following homomorphism is well-defined with kernel containing $C_G(S)$: \[\chi_S: G \to k^\times,\, x \mapsto \begin{cases} -1 & x \not\in N_G(A) \\ 1 & x \in N_G(A)\end{cases}\] It follows that the pair $(\chi_P)_{P \in s_p(G)}$ with $\chi_P \equiv 1$ for $P \neq S$ and $\chi_S$ as above is $G$-stable and satisfies the coherence conditions for all cyclic subgroups of $S$. Therefore, it remains to show that $G \setminus N_G(A)$ is not determined by coherence conditions, and we do this by explicitly determining the coherence conditions for $G$ arising from each cyclic subgroup of $S$. For $A$ and $B$, $N_G(A) = N_G(B) = S \rtimes (\Aut(A) \times \Aut(B))$, so any coherence conditions corresponding to $A$ or $B$ only affect $N_G(A)$. Let $C := \langle ab^n\rangle$ for $n \in \{1, \dots, p-1\}$, then $N_G(C) \cong (S \rtimes \Aut(C)) \rtimes C_2$, where $C_2$ is generated by an involution $\sigma$ satisfying ${}^\sigma a = b^n$.

        The coherence condition states that for $x = x_p x_{p'}\in N_G(C)$, we have an equality $\chi_C(xC) = \chi_{P\langle x_p\rangle}(x P\langle x_p \rangle)$. If $x_p \in C$, then the coherence condition is therefore trivial. On the other hand, if $x_p \in S\setminus C$, we know $x_{p'} \in C_G(x_p)$ by definition) and $x_{p'} \in N_G(C)$ since $x, x_p \in N_G(C)$ as well. Thus, $x_{p'} \in C_{N_G(C)}(x_p)$.

        We claim that when $x_p \in S\setminus C$, then $C_{N_G(C)}(x_p) = S$. Indeed, if $x_p \in A$ or $x_p \in B$, then $C_G(x_p) \cong S \rtimes C_{p-1}$ with $C_{p-1} = \Aut(B)$ or $C_{p-1} = \Aut(A)$ respectively, and no nontrivial element of $\Aut(A)$ or $\Aut(B)$ centralizes $C$. Otherwise, if $x_p \in \langle ab^m\rangle$ for some $m \in \{1, \dots, p-1\}$ and $m \neq n$, then, $C_G(x_p) \cong S \rtimes C_2$, where $C_2$ is generated by an involution $\sigma$ satisfying ${}^\sigma a = b^m$. However, because $C = \langle a b^n \rangle$, it follows that $\sigma$ does not normalize $C$, thus in all cases we have $C_{N_G(C)}(x_p) = S$.

        Therefore, for any $x = x_px_{p'} \in N_G(C)$, if $C\langle x_p\rangle = S$, then $x_{p'} = 1$, since $x_{p'} \in C_{N_G(C)}(x_p)$ and is a $p'$ element. Thus, the coherence condition implies no new information. As we have considered every cyclic subgroup of $S$, we conclude that the coherence conditions do not imply anything about $x \not\in N_G(A)$. Therefore, $\calR_G$ has at least three generators, as desired, and $\Lambda$ is not surjective.
    \end{proof}

    The previous example generalizes to groups with arbitrary large elementary abelian Sylow $p$-subgroups.

    \begin{theorem}
        Let $p$ be an odd prime, $k$ a field of characteristic $p$, and $n \geq 2$ an integer. Then there exists a finite group $G$ with $p$-rank $n$ such that $\Lambda: \calE_k(G) \to O(T(kG))$ is not surjective.
    \end{theorem}
    \begin{proof}
        We set $G := (C_p \rtimes C_{p-1})^n \rtimes S_n$, where each copy of $C_{p-1}$ acts faithfully on the corresponding copy of $C_p$ and $S_n$ permutes the copies of $C_p \rtimes C_{p-1}$ in the obvious way. We regard $G$ as a subgroup of $S_{np}$, with the copies of $C_p$ generated by $(1,2,\dots, p), \dots, ((n-1)p + 1, \dots, np)$. First, it is clear that we have a $p$-subgroup (not necessarily Sylow) $E \cong C_p^n$, and straightforward computation shows that for every cyclic subgroup $C$ of order $p$ of $E$, we have $N_G(C)/C_G(C) \cong C_{p-1}$. There are $n$ conjugacy classes of cyclic subgroups of order $p$ in $E$, corresponding to cycle type via the embedding $G \hookrightarrow S_{np}$. Therefore, to show $\Lambda$ is not surjective for $G$, it suffices to show that that $\calR_G$ has at least one more generator.

        Let $ P:= \langle (1,\dots,p), (p+1, \dots, 2p)\rangle \leq S_{np}$, and set $H := N_{S_{2p}}(P)$. That is, $H$ is isomorphic to the group considered in Theorem \ref{prop:prk2}. Explicitly, $H \cong (C_p \times C_{p-1})^2 \rtimes C_2$, and $P$ the unique Sylow $p$-subgroup of $H$. First, observe we have an isomorphism $N_G(P)/C_G(P) \cong N_H(P)/C_H(P)$, hence the hom-sets associated to $P$ for $\calR_G$ and $\calR_H$ coincide. Moreover, for any subgroup $C $ of $ P$, the values of any $\chi_C \in \Hom(N_G(C)/C_G(C), k^\times)$ are determined entirely on values in $N_H(C)$, since $N_G(C) = N_H(C)C_G(C)$. Therefore, any coherence conditions for the tuple $\Hom(N_G(P)/C_G(P),k^\times)$ in $\calR_G$ arise from coherence conditions in $\calR_H$. It follows from Proposition \ref{prop:prk2} that $\calR_G$ has another generator, corresponding to the following $\chi_P \in \Hom(N_G(P)/C_G(P),k^\times)$: \[\begin{cases}
            \chi_P(x) = -1 & x \in N_G(P) \setminus N_G(A)\\ 1 & \text{otherwise,}
        \end{cases}\] where $A$ denotes the cyclic subgroup generated by $(1,2,\dots, p)$ under the inclusion $G \hookrightarrow S_{np}$.
    \end{proof}

    We expect that for ``most'' finite groups $G$ with $p$-rank at least 2, the Lefschetz homomorphism is not surjective.

    \section{Towards the kernel of the Lefschetz homomorphism} \label{sec:kernel}

    We conclude with a few results towards deducing the kernel of the Lefschetz homomorphism. The following proposition was partially observed by Bachmann \cite[5.50(iv)]{Bac16} in the context of invertible Artin motives and with restricted hypotheses. We refer the reader to \cite{BG23} for an explanation of the correspondence between Artin motives and the bounded homotopy category of $p$-permutation modules (and the corresponding ``big" category, the derived category of permutation modules).

    \begin{theorem}\label{thm:kerforcyclics}
        Assume $p$ is an odd prime. Let $G$ be a finite group with a cyclic Sylow $p$-subgroup $S$, and $C \in \calT\calE(G)$ be an endotrivial complex with trivial homology. We have that $C \in \ker(\Lambda)$ if and only if $h_C$ admits only even values and for all nontrivial $p$-subgroups $P$ of $S$, the congruence $h_C(1) \equiv h_C(P) \mod 2\Phi(S)$ holds.
    \end{theorem}
    \begin{proof}
        We rely on Theorem \ref{thm:etrivtodecomp} and the description of $\calR_G$ determined via Proposition \ref{prop:nocoherenceforcyclics} and Proposition \ref{prop:automizers}. For each proper subgroup $P$ of $S$, let $C_P$ denote the endotrivial complex with trivial homology and $h_{C_P}(Q) = 2$ for all $Q < P$ and $h_{C_P}(Q) = 0$ otherwise. Set $C_S = k[1]$. Then $\{C_P\}_{1 \leq P \leq S}$ is a $\Z$-basis of $\calT\calE(G)$, and we may write $[C] = a_1[C_1] + \dots + a_S[C_S] \in \calT\calE(G)$. A routine computation using the h-mark functions associated to each $C_P$ shows that the congruence $h_C(1) \equiv h_C(P) \mod 2\Phi(S)$ holds for all nontrivial $p$-subgroups $P$ of $S$ if and only if $a_Q$ is a multiple of $\Phi(S)$ for all proper subgroups $Q$ of $S$.

        First, suppose $h_C$ admits only even values and for all nontrivial $p$-subgroups $P$ of $S$, the congruence $h_1(C) \equiv h_P(C) \mod 2\Phi(S)$ holds. Then equivalently, under the decomposition $[C] = a_1[C_1] + \dots + a_S[C_S] \in \calT\calE(G)$, $a_Q$ is a multiple of $\Phi(S)$ for all proper subgroups $Q$ of $S$. Since $\calR_G$ is isomorphic to a direct product of finite groups by Proposition \ref{prop:nocoherenceforcyclics}, each of order $\Phi(S)$ by Proposition \ref{prop:automizers}, the image of $C$ is trivial in $\calR_G$. Likewise, since $h_C$ admits only even values, the image of $C$ is trivial in $(B(S)^G)^\times$ as well, hence $C \in \ker(\Lambda)$.

        Conversely, suppose $C \in \ker(\Lambda)$. It is easy to see that necessarily $h_C$ must be even-valued. We induct on the order of $S$, in a similar manner to Theorem \ref{thm:cyclicsurj}. Again, it follows that $C \in \ker(\Lambda)$ if and only if $\Res^G_{N_G(H)} \in \ker(\Lambda)$, where $H$ denotes a cyclic $p$-subgroup of order $p$, since $N_G(H) \geq N_G(P)$ for any other $p$-subgroup $P \geq H$. Therefore, we may assume without loss of generality that $H$ is normal. From the proof of Theorem \ref{thm:cyclicsurj}, we have a decomposition $\calT\calE(G) = \Inf^G_{G/H}\calT\calE(G/H) \times \partial\calT\calE(G)$, with image in $\calR_G$ that intersects trivially. Let $[C] = a_1[C_1] + \dots + a_S[C_S] \in \calT\calE(G)$. We have that $C_1 \in \partial\calT\calE(G)$ and $C_P \in \Inf^G_{G/H}\calT\calE(G/H)$ for all nontrivial subgroups $P$ of $S$. Since the images of $\Inf^G_C\calT\calE(G/C)$ and $\partial\calT\calE(G)$ intersect trivially, both $a_1[C_1]$ and $a_H[C_H]+ \cdots + a_S[C_S]$ belong to $\ker(\Lambda)$. By induction, it follows that $a_P$ is a multiple of $\Phi(S)$ for all nontrivial $P$.

        On the other hand, by \cite[Lemma 5.8]{M24a}, it follows that $C_1^{\otimes a_1}$ is homotopy equivalent to an indecomposable chain complex contains only one non-projective indecomposable module $M$ of $k$-dimension one. By the Green correspondence, $\Res^G_{N_G(S)} M$ is the trivial module if and only if $M$ is the trivial module. Since the image of $C_1^{\otimes a_1}$ in $\calR_G$ is necessarily trivial, $\Res^G_{N_G(S)}(M)$ is trivial, hence $M$ is trivial. Therefore $C_1^{\otimes a_1}$ is homotopy equivalent to a truncated projective resolution of the trivial $kG$-module, hence it has length a multiple of the periodicity of $G$, which is $2\Phi(S)$. Thus, $a_P$ is a multiple of $\Phi(S)$ for all proper subgroups $P$ of $S$, so $h_1(C) \equiv h_Q(C) \mod 2\Phi(S)$ holds for all nontrivial $p$-subgroups $Q$ of $S$, as desired.
    \end{proof}

    \begin{theorem}\label{thm:kerforarb}
        Let $C \in \calT\calE(G)$ be an endotrivial complex with trivial homology. If $p = 2$, then $C \in \ker(\Lambda)$ if and only if $h_C$ is even-valued. If $p$ is odd and $C \in \ker(\Lambda)$, then $h_C$ admits only even values and for all pairs $K,H$ of subgroups of $G$ satisfying $[H:K] = p$ and $K \leq H \leq N_G(K)$, the following congruence holds: \[h_K(C) \equiv h_H(C) \mod 2|O_{p'}(\Aut_{N_G(K)}(H))|.\]
    \end{theorem}
    \begin{proof}
        If $p=2$, then if $C \in \calT\calE(G)$, $\calH_C(P) = 1$ for any $p$-subgroup $P$ of $G$. Therefore, it easily follows from Theorem \ref{thm:etrivtodecomp} that $h_C$ is even valued if and only if $C \in \ker(\Lambda)$.

        If $p$ is odd, similarly it is clear that $h_C$ admits only even values. For the second term, it suffices to show it for the pair $1,H$, with $H$ a subgroup of $G$ of order $p$ by replacing $C$ with $C(K)$, as $C(K) \in \ker(\Lambda)$ as well. For every subgroup $H'$ of $G$ with Sylow $p$-subgroup $H$, we have a congruence $h_C(1) \equiv h_C(H) \mod 2\Phi(H')$ by Theorem \ref{thm:kerforcyclics}. Therefore, it remains to determine the least common multiple of all such $2\Phi(H') = 2|\Aut_{H'}(H)|.$ We claim this is the order of the largest $p'$-subgroup of $\Aut_G(H)$. Indeed, every element of $\Aut_{H'}(H)$ is a $p'$-element of $\Aut_G(H)$ and conversely, every $p'$-element $c_g \in \Aut_G(H)$ falls in some $\Aut_{H'}(H)$ by setting $H' := H \rtimes \langle g\rangle$ (if $g$ is not a $p'$-element, since $c_g$ has $p'$-order $g$ may be replaced with a suitable $p'$-element). Thus, we conclude $h_1(C) \equiv h_H(C) \mod 2|O_{p'}(\Aut_G(H))|$.
    \end{proof}

    The question of whether the converse of Theorem \ref{thm:kerforarb} holds for odd primes $p$ remains open. Given any endotrivial complex $C$, if $h_C(P) = 0$ for all cyclic $p$-subgroups $P$ of $G$, then $h_C \equiv 0$. We ask if the same should hold for $\calH_C$: if $\calH_C(P) = 1 \in \Hom(N_G(P)/P, k^\times)$ for all cyclic $p$-subgroups $P$ of $G$, does $\calH_C \equiv 1$? If the answer to this question is affirmative, a complete determination of $\ker(\Lambda)$ follows.

    \bibliography{bib}
    \bibliographystyle{alpha}

\end{document}